\documentclass[a4paper,12pt,reqno]{amsart}
\usepackage{amsfonts}
\usepackage{amsmath}
\usepackage{amssymb}
\usepackage[a4paper]{geometry}
\usepackage{mathrsfs}
\usepackage[colorlinks]{hyperref}

\numberwithin{equation}{section}
\theoremstyle{plain}
\newtheorem{thm}{Theorem}[section]

\newtheorem{cor}[thm]{Corollary}

\theoremstyle{definition}
\newtheorem{defn}[thm]{Definition}

\def\R{\mathcal R}
\def\Q{\Omega}
\def\G{{\mathbb G}}
\def\L{\Delta_{\mathbb{H}}}
\def\Rn{\mathbb{R}^{n}}
\def\H{\mathbb{H}^{n}}
\def\Ls{\mathcal{L}}
\def\N{\mathcal{R}^{\frac{a}{\nu}}}
\def\na{\nabla_{\mathbb{H}}}

\begin{document}

\title [Liouville theorems on the Heisenberg groups]
{Liouville theorems for Kirchhoff-type hypoelliptic Partial Differential Equations and systems. I.  Heisenberg group}

\author[A. Kassymov]{Aidyn Kassymov}
\address{
  Aidyn Kassymov:
  \endgraf
  Department of Mathematics: Analysis, Logic and Discrete Mathematics
  \endgraf
  Ghent University, Belgium
   \endgraf
  and
  \endgraf
  Institute of Mathematics and Mathematical Modeling
  \endgraf
  Almaty, Kazakhstan
  \endgraf
  {\it E-mail address} {\rm kassymov@math.kz} and {\rm aidyn.kassymov@ugent.be}}

  \author[M. Ruzhansky]{Michael Ruzhansky}
\address{
	Michael Ruzhansky:
	 \endgraf
  Department of Mathematics: Analysis, Logic and Discrete Mathematics
  \endgraf
  Ghent University, Belgium
  \endgraf
  and
  \endgraf
  School of Mathematical Sciences
    \endgraf
    Queen Mary University of London
  \endgraf
  United Kingdom
	\endgraf
  {\it E-mail address} {\rm michael.ruzhansky@ugent.be}
}

\author[N. Tokmagambetov]{Niyaz Tokmagambetov}
\address{
  Niyaz Tokmagambetov:
  \endgraf
  Department of Mathematics: Analysis, Logic and Discrete Mathematics
  \endgraf
  Ghent University, Belgium
  \endgraf
  and
  \endgraf
  Al-Farabi Kazakh National University
  \endgraf
   71 Al-Farabi avenue
   \endgraf
   050040 Almaty
   \endgraf
   Kazakhstan
   \endgraf
  {\it E-mail address} {\rm niyaz.tokmagambetov@ugent.be}
  }

\author[B. Torebek]{Berikbol Torebek}
\address{
	Berikbol Torebek:
   \endgraf
  Department of Mathematics: Analysis, Logic and Discrete Mathematics
  \endgraf
  Ghent University, Belgium
  \endgraf
  and
  \endgraf
 Institute of Mathematics and Mathematical Modeling
   \endgraf
  Almaty, Kazakhstan
   \endgraf
  {\it E-mail address} {\rm berikbol.torebek@ugent.be}}

\date{\today}

\thanks{The authors were supported by the FWO Odysseus 1 grant G.0H94.18N: Analysis and Partial Differential Equations. MR was supported in parts by the EPSRC Grant EP/R003025/2 and  by the Methusalem programme of the Ghent University Special Research Fund (BOF) (Grant number 01M01021).}



\begin{abstract}
In this paper, we show the nonexistence results for the Kirchhoff elliptic, parabolic, and hyperbolic type equations on the Heisenberg groups.  Also, the pseudo-parabolic and pseudo-hyperbolic equations of the Kirchhoff-type are under consideration. To prove these results we use the  test function method. In addition, the analogous results are transferred to the cases of systems. Also, we give some examples of non-local nonlinearities.
\end{abstract}

\maketitle

\tableofcontents
\section{Introduction}

\subsection{Heisenberg groups}
Let us briefly recall the definition of the Heisenberg group. For $n\in\mathbb N$, $\mathbb{H}^{n}:=(\mathbb{R}^{2n+1}, \cdot)$ is called the Heisenberg group if the group law is given by
\begin{equation}
\begin{split}
x \cdot x'=(\xi_{1}+\xi_{2},\tilde{\xi}_{1}+\tilde{\xi}_{2},\tau_{1}+\tau_{2}+2(\xi_{1}\tilde{\xi}_{2}-\xi_{2}\tilde{\xi}_{1})),
\end{split}
\end{equation}
for any $x=(\xi_{1},\tilde{\xi}_{1},\tau_{1})$ and $x'=(\xi_{2},\tilde{\xi}_{2},\tau_{2}),$ where $\xi_{1}, \xi_{2}, \tilde{\xi}_{1}, \tilde{\xi}_{2}\in \Rn$ and $\tau_{1}, \tau_{2}\in \mathbb{R}$.
The family of dilations has the following form
\begin{equation}
\delta_{\lambda}(x):=(\lambda \xi, \lambda \tilde{\xi}, \lambda^{2}\tau),\,\,\,\,\,\forall \lambda>0.
\end{equation}
So, the homogeneous dimension of $\mathbb{H}^{n}$ is $Q=2n+2$ and the topological dimension is $2n+1$.
The Lie algebra $\mathfrak{h}$ of the left-invariant vector fields on the Heisenberg group $\mathbb{H}^{n}$ is spanned by
$$
X_{i}=\frac{\partial}{\partial\xi_{i}}+2\tilde{\xi}_{i}\frac{\partial}{\partial \tau}, \,\,\, Y_{i}=\frac{\partial}{\partial\widetilde{\xi_{i}}}-2\xi_{i}\frac{\partial}{\partial \tau},\,\,\,x=(\xi,\tilde{\xi},\tau),
$$
for $i=1,\ldots,n,$ with their non-zero commutator
$$[X_{i},Y_{i}]=-4\partial_{\tau}.$$
The horizontal gradient on $\mathbb{H}^{n}$ is given by
$$\na:=(X_{1},\ldots,X_{n},Y_{1},\ldots,Y_{n}).$$
So, the Kohn-Laplacian or sub-Laplacian is given by
$$\L:=\sum_{i=1}^{n}\left(X_{i}^{2}+Y_{i}^{2}\right).$$

Finally, we define the (Kaplan) distance on the Heisenberg group by
$$
|x|_{H}^{4}=\tau^{2}+\sum_{i=1}^{n}\left(|\xi|^{2}+|\tilde{\xi}|^{2}\right)^{2},
$$
for $x = (\xi, \tilde{\xi}, \tau)\in \H.$ For further details, we refer to the book \cite{FR16}.

Now let us denote by $S^2_1(\H)$ the Sobolev space  on the Heisenberg group:
\begin{equation}
  S^2_1(\H):=\{u:u\in L^{2}(\H),\,\,|\na u|\in L^{2}(\H)\}.
\end{equation}
Similarly, let us denote $S^{2}_{2}(\H)$ in the following form:
\begin{equation}
  S^{2}_{2}(\H):=\{u:u\in L^{2}(\H),\,\,|\na u|\in L^{2}(\H),\,\,|\L u|\in L^{2}(\H)\}.
\end{equation}

\subsection{Historical background} 
In the celebrated paper \cite{Fuj66}, Fujita studied the following nonlinear heat equation
\begin{equation}\label{EQ: 01}
\begin{cases}
u_{t}(x,t)-\Delta u(x,t)=u^{1+p}(x, t), \,\,\, (x,t)\in \mathbb{R}^{N}\times (0,\infty), \\ u(x,0)=u_{0}(x)\geq0, \,\,\, x\in \mathbb{R}^{N}.
 \end{cases}
\end{equation}
He showed that if $0<p<\frac{2}{N}$ then a solution of the problem \eqref{EQ: 01} blows up in finite time  for  $N>2$ while being globally well-posed for $p>\frac{2}{N}$, causing to call this critical exponent ``the critical exponent of Fujita" or just ``Fujita's exponent". One of the further generalisations of the problem \eqref{EQ: 01} is considering the fractional Laplacian $(-\Delta)^{s}$ instead of the classical one $(-\Delta)$. That is, in \cite{GK99, S75} the authors considered the Cauchy problem
\begin{equation*}
 \begin{cases}
   u_{t}(x,t)+(-\Delta)^{s}u(x,t)=a(x,t)|u(x,t)|^{1+p},\,\,\,(x,t)\in \mathbb{R}^{N}\times (0,\infty),\\
   u(x,0)=u_{0}(x)\geq0,\,\,\,x\in \mathbb{R}^{N},
 \end{cases}
\end{equation*}
for $s>0$.

Consequently, in \cite{PohVer}, V\'{e}ron and Pohozhaev found a critical exponent for the following non-linear diffusion equation with the Kohn-Laplacian on Heisenberg groups
$$
\frac{\partial u(x,t)}{\partial t}-\L u(x,t)=|u(x,t)|^{p},\,\,\,\,\,(x,t)\in\H\times(0,+\infty).
$$

Also, the critical exponents for other equations with the Kohn-Laplacian on the Heisenberg groups were derived in \cite{AAK1, AAK2,DL1,JKS}.
In addition, in \cite{RY} the authors found the Fujita exponent on general unimodular Lie groups. 

In the paper \cite{Kato}, T. Kato showed the nonexistence of the global solutions  for the Cauchy problem for a  non-linear wave equation. This work was a pioneering work in this field for the hyperbolic type equations. The corresponding order is called the Kato exponent.  As a generalisation to the Heisenberg group $\H$, V\'{e}ron and Pohozhaev in \cite{PohVer} considered the following problem
$$ \frac{\partial^{2} u(x,t)}{\partial t^{2}}-\L u(x,t)=|u(x,t)|^{p},\,\,\,\,\,(x,t)\in\H\times(0,+\infty),$$
with the Cauchy data
$$u(x,0)=u_{0}(x),\,\,\,\,u_{t}(x,0)=u_{1}(x),$$
where $p>1$.
Also, in the work \cite{AAK1} a fractional analogue was studied, namely, the following non-linear wave equation 
$$ 
\frac{\partial^{2} u(x,t)}{\partial t^{2}}+(-\Delta_{\mathbb{H}})^{s}u(x,t)=|u(x,t)|^{p},
\,\,\,\,\,(x,t)\in\H\times(0,+\infty),
$$
with the Cauchy data
$$u(x,0)=u_{0}(x),\,\,\,\,u_{t}(x,0)=u_{1}(x),$$
where $(-\Delta_{\mathbb{H}})^{s}$ is the fractional Kohn-Laplacian, $s\in(0,1),\,\,\,p>1$.
The wave equation with nonlinear damping on the Heisenberg and graded groups was studied in \cite{RT18}.  In \cite{Geor1, Geor}, one found a critical exponents for the semilinear heat equation and wave equation with damping term on the Heisenberg group.
Also, comparison principles for heat equations with Rockland operators on general graded groups were considered  in \cite{RS1} and \cite{RY1}.

In \cite{AR73}, under certain assumptions on $f$ for the semilinear equation
\begin{equation}
\label{eqar}
\begin{cases}
-\Delta u=f(x,u),\,\,x\in\Omega\subset \mathbb{R}^{n},\\
u(x)=0,\,\,x\in\partial \Omega,
\end{cases}
\end{equation}
Ambrosetti and Rabinowitz proved the existence of solutions by the mountain pass theorem, leading to a number of extensions and generalisations. Indeed, the authors proved the existence results of the weak solutions of \eqref{eqar} in the model case $f(x,u)=a(x)|u|^{q-1}u$, where $1<q<\frac{n+2}{n-2}$.

By the fibering method the authors of the work \cite{DP97} proved the existence of the weak solutions to
\begin{equation}
\begin{cases}
-\Delta_{p}u(x)=\lambda|u|^{p-2}u+f(x)|u|^{q-2}u,\,\,x\in\Omega\subset \mathbb{R}^{n},\\
u(x)=0,\,\,x\in\partial \Omega,
\end{cases}
\end{equation}
where $p,\lambda,q\in\mathbb{R}$ and $1<p<q<p^{*}=\frac{np}{n-p}$. 

In the work \cite{C95}, the author considered and showed that there is at least one weak solution to the following problem on Heisenberg groups
\begin{equation}
 \begin{cases}
 -\L u+au=|u|^{\frac{Q+2}{Q-2}},\,\,x\in\Omega\subset \H,\\
 u(x)=0,\,\,x\in\partial \Omega.
 \end{cases}
\end{equation}
Also, see \cite{GL}.

An important generalisation of the semilinear elliptic type equations is the Kirchhoff type equations, which arise in the description of nonlinear vibrations of an elastic string. Also, the study of Kirchhoff-type problems has been receiving considerable attention in recent years, see \cite{APS, AP, HY, KTT, KTT1, MR1, MR2,  MR3,  Rad, Rad1} and \cite{XZF}. This interest arises from their contributions to the modeling of many physical and biological phenomena. The main aim of this note is to extend the above results to the Heisenberg group case, or more general cases. Thus, we will consider nonlinear Kirchhoff type problems on the Heisenberg group. 
Note that there is already a number of results related to the non-existence of  solutions to the semilinear equations, mentioning only few of the pioneering works on the Heisenberg group, for instance, \cite{AAK1}, \cite{AAK2}  and  \cite{PohVer}. In the present paper, we will use the test function approach, and new interesting results on the Heisenberg group will be proved. Some results obtained here can be  extended to the general stratified Lie groups and graded Lie groups.


In this paper, we also study the Kirchhoff type nonlinear systems of elliptic, parabolic, and hyperbolic type equations. In all of our equations the coefficient which expresses the Kirchhoff type non-linearity will be denoted by $M$. In each individual case we will ask different properties on $M$. But the following particular cases could be a few examples of the function $M$:

    a) $M=1$,
	
    b) $M=\frac{1}{1+a\int_{\H}|\na u|^{2}dx+b\int_{\H}|\L u|^{2}dx}$, where $a,b$ are some real constants.
	
    c) $M=(1+c|t|^{\gamma})^{-\beta}$, where $ c, \gamma$ are some real constants and $\beta>0$.

    d)  $M=\exp(-(a\int_{\H}|\na u|^{2}dx+b\int_{\H}|\L u|^{2}dx))$, where $a,b$ are real constants.

    e) $M=\log\left(\frac{a+b\int_{\H}|\na u|^{2}dx+c\int_{\H}|\L u|^{2}dx}{a_{1}+b_{1}\int_{\H}|\na u|^{2}dx+c_{1}\int_{\H}|\L u|^{2}dx}\right)$, $a,b,c,a_{1},b_{1},c_{1}$ are some real constants.

\section{Kirchhoff type elliptic equations and systems}
\label{S-E}

In this section we find the critical exponents for the Kirchhoff type elliptic equations on the Heisenberg group.

\subsection{Elliptic type equations}
Let us consider the following Kirchhoff type elliptic equation
\begin{equation}\label{ell}
-M\left(\int_{\H}|\na u(x)|^{2}dx,\int_{\H}|\L u(x)|^{2}dx\right) \L u(x)= |u(x)|^{p},   \,\, x \in \H,\,\, p > 1,
\end{equation}
where $u\in C^{2}(\H)$ and $M:\mathbb{R}_{+}\times\mathbb{R}_{+}\rightarrow\mathbb{R}$ is a bounded function such that
\begin{equation}\label{Mel}
0<M(\cdot,\cdot)\leq C_{0}.
\end{equation}

Let us introduce a notion of the weak solution to the equation \eqref{ell}.
\begin{defn}\label{defnweakel}
We say that $u\in S^{2}_{2}(\H)\cap L^{p}(\H)$ with $p>1$ is a  weak solution of the Kirchhoff type elliptic equation \eqref{ell} on $\H$, if the following identity
\begin{equation*}
\int_{\H}|u|^p\varphi dx
=  - \int_{\H}M\left(\int_{\H}|\na u|^{2}dx,\int_{\H}|\L u|^{2}dx\right) u\L \varphi dx
\end{equation*}
holds true, for any test function $0\leq\varphi\in C_{0}^{2}(\H)$.
\end{defn}

Then we have the following result.
\begin{thm}\label{ellthm}
Assume that the function $M(\cdot,\cdot)$ satisfies \eqref{Mel} and
\begin{equation}
1<p\leq\frac{Q}{Q-2},
\end{equation}
with $Q=2n+2$ the homogeneous dimension of $\H$.
Then the problem \eqref{ell} does not have a global non-trivial weak solution.
\end{thm}

\begin{proof} In the proof we widely use the method of test functions. We start from Definition \ref{defnweakel}. From the identity \eqref{ell} one has
\begin{equation}
\label{elldoyou}
\begin{split}
\int_{\H}|u|^p\varphi dx
&=  - \int_{\H} M\left(\int_{\H}|\na u|^{2}dx,\int_{\H}|\L u|^{2}dx\right) u\L \varphi dx
\\&\leq\left|\int_{\H}M\left(\int_{\H}|\na u|^{2}dx,\int_{\H}|\L u|^{2}dx\right) u\L \varphi dx\right|
\\&\leq\int_{\H} \left|M\left(\int_{\H}|\na u|^{2}dx,\int_{\H}|\L u|^{2}dx\right)\right| |u||\L \varphi| dx
\\&\leq C_{0} \int_{\H}|u||\L \varphi| dx.
\end{split}
\end{equation}
Now, we choose the following test function
\begin{equation}
\label{testfuncell}
\varphi_{R} (x) = \Phi \left(\frac{|\xi|^{4}+|\tilde{\xi}|^{4}+\tau^{2}}{R^{4}}\right), \,\,x=(\xi,\tilde{\xi},\tau)\in\H,\,\,R>0,
\end{equation}
with $\Phi\in C^{\infty}_{0}(\mathbb{R}_{+})$ satisfying the following structure
\begin{equation}\label{testell}
\Phi(r) =
    \begin{cases}
1, & \text{if $ 0\leq r \leq 1 $,} \\
\searrow, & \text{if $1< r\leq 2$,}\\
0, & \text{if $r>  2$,}
\end{cases}
\end{equation}
with $\text{supp}\,\Phi=[0,2]$. We note that $\text{supp}(\varphi_{R})$ and $\text{supp}(\L\varphi_{R})$ are subsets of $\Omega_{1}$, where
\begin{equation*}
\Omega_{1}:=\{x=(\xi,\tilde{\xi},\tau)\in\H:|\xi|^{4}+|\tilde{\xi}|^{4}+\tau^{2}\leq 2R^{4}\}.
\end{equation*}

By using $\varepsilon$-Young's inequality in \eqref{elldoyou}, we obtain
\begin{equation}
\begin{split}
\int_{\H}|u|^p\varphi_{R} dx&\leq  C_{0}\int_{\H} |u||\L \varphi_{R}| dx\\&
\stackrel{\eqref{testell}}=C_{0}\int_{\Omega_{1}} |u||\L \varphi_{R}| dx
\\&=C_{0} \int_{\Omega_{1}}|u|\varphi_{R}^{\frac{1}{p}}\varphi_{R}^{-\frac{1}{p}}|\L \varphi_{R}| dx
\\& =C_{0} \left(\frac{\int_{\Omega_{1}}|u|^{p}\varphi_{R} dx}{\varepsilon ^{p}p}+\frac{(p-1)\varepsilon^{\frac{p}{p-1}}\int_{\Omega_{1}}\varphi_{R}^{-\frac{1}{p-1}}|\L \varphi_{R}|^{\frac{p}{p-1}} dx}{p}\right)
\\&=\frac{C_{0}}{p\varepsilon^{p}} \left(\int_{\Omega_{1}}|u|^{p}\varphi_{R} dx+(p-1)\varepsilon^{p+\frac{p}{p-1}}\int_{\Omega_{1}}\varphi_{R}^{-\frac{1}{p-1}}|\L \varphi_{R}|^{\frac{p}{p-1}} dx\right).
\end{split}
\end{equation}
By choosing $\varepsilon$ such that $C_{0}<p\varepsilon^{p}$   we get
\begin{equation}
\label{26}
\begin{split}
\int_{\H}|u|^{p}\varphi_{R} dx\leq C\left(\int_{\Omega_{1}}\varphi_{R}^{-\frac{1}{p-1}}|\L \varphi_{R}|^{\frac{p}{p-1}} dx\right).
\end{split}
\end{equation}
Let us denote by 
$$\rho=\frac{|\xi|^{4}+|\tilde{\xi}|^{4}+\tau^{2}}{R^{4}}.$$
Then from \cite{PohVer}, one has the following property
\begin{equation}\label{poxxell}
\begin{split}
\L\varphi_{R}&(x)=\frac{4(n+4)}{R^{4}}\left(|\xi|^{2}+|\tilde{\xi}|^{2}\right)\Phi'\left(\rho\right)
\\&+\frac{16}{R^{8}}\left((|\xi|^{6}+|\tilde{\xi}|^{6})+2\tau(|\xi|^{2}-|\tilde{\xi}|^{2})\xi\cdot\tilde{\xi}+\tau^{2}(|\xi|^{2}+|\tilde{\xi}|^{2})\right)\Phi''\left(\rho\right).
\end{split}
\end{equation}
Taking into account \eqref{poxxell}, we obtain
\begin{equation}\label{ocenell123}
|\L \varphi_{R}|\leq CR^{-2}.
\end{equation}
Let us denote 
$$\Theta:=\{y=(\overline{\xi},\widehat{\xi},\tilde{\tau})\in \H: |\overline{\xi}|^{4}+|\widehat{\xi}|^{4}+\tilde{\tau}^{2}\leq 2\},$$
and 
$$\mu:=|\overline{\xi}|^{4}+|\widehat{\xi}|^{4}+\tilde{\tau}^{2}.$$
By changing variables $R\overline{\xi}=\xi$, $R\widehat{\xi}=\tilde{\xi}$, $R^{2}\tilde{\tau}=\tau$   in \eqref{26} and by using \eqref{ocenell123}, we get
\begin{equation}
\label{210}
\begin{split}
&\int_{\H}|u|^{p}\varphi_{R} dx\leq C\left(\int_{\Omega_{1}}\varphi_{R}^{-\frac{1}{p-1}}|\L \varphi_{R}|^{\frac{p}{p-1}} dx\right)
\\&\leq C R^{-\frac{2p}{p-1}+Q}\left(\int_{\Theta}(\Phi\circ\mu)^{-\frac{1}{p-1}} dy\right)\\&
\leq  CR^{-\frac{2p}{p-1}+Q}.
\end{split}
\end{equation}
Let us consider the case $p<\frac{Q}{Q-2},$ and we note that $-\frac{2p}{p-1}+Q<0$ is equivalent to $p<\frac{Q}{Q-2}.$ Then letting $R\rightarrow \infty$ in \eqref{210}, we arrive at
\begin{equation}\label{osnellrav}
\begin{split}
\int_{\H}|u|^{p} dx\leq C \lim_{R\rightarrow \infty}R^{-\frac{2p}{p-1}+Q}=0,
\end{split}
\end{equation}
concluding that $u=0$. This is a contradiction.

Let us consider the case $p=\frac{Q}{Q-2}>1$. In this case, from \eqref{osnellrav}, we get 
\begin{equation}\label{upell}
\begin{split}
\int_{\H}|u|^{p} \varphi_{R} dx\leq C,
\end{split}
\end{equation}
where $C>0$.

Let us introduce the following domain:
\begin{equation}
    \overline{\Omega}_{R}:=\{x=(\xi,\tilde{\xi},\tau)\in\H:R^{4}\leq |\xi|^{4}+|\tilde{\xi}|^{4}+\tau^{2}\leq 2R^{4},
\end{equation}
then by combining \eqref{upell} and the Lebesgue dominated convergence theorem, we have
\begin{equation}\label{00}
    \lim_{R\rightarrow \infty}\int_{\overline{\Omega}_{R}}|u|^{p}\varphi_{R} dx=0.
\end{equation}

Let us denote by 
$$\Theta_{1}:=\{y=(\overline{\xi},\widehat{\xi},\tilde{\tau})\in \H:1\leq |\overline{\xi}|^{4}+|\widehat{\xi}|^{4}+\tilde{\tau}^{2}\leq 2\},$$
By using H\"{o}lder's inequality in \eqref{elldoyou}, we get

\begin{equation*}
\begin{split}
\int_{\H}|u|^p\varphi_{R} dx&\leq C_{0} \int_{\Omega_{1}}|u||\L \varphi_{R}| dx\\&
\leq C_{0} \int_{\Omega_{1}\setminus\overline{\Omega}_{R}}|u||\L \varphi_{R}| dx+C_{0} \int_{\overline{\Omega}_{R}}|u||\L \varphi_{R}| dx
\\&
\stackrel{\eqref{testell}}=C_{0} \int_{\overline{\Omega}_{R}}|u||\L \varphi_{R}| dx
\\&=C_{0} \int_{\overline{\Omega}_{R}}|u|\varphi_{R}^{\frac{1}{p}}\varphi_{R}^{-\frac{1}{p}}|\L \varphi_{R}| dx
\\&\leq C_{0}\left(\int_{\overline{\Omega}_{R}}|u|^{p}\varphi_{R} dx\right)^{\frac{1}{p}}\left(\int_{\overline{\Omega}_{R}}\varphi_{R}^{-\frac{1}{p-1}}|\L \varphi_{R}|^{\frac{p}{p-1}} dx\right)^{\frac{p-1}{p}}\\&
\stackrel{\eqref{210}}=  C_{0}R^{-2+\frac{Q(p-1)}{p}}\left(\int_{\overline{\Omega}_{R}}|u|^{p}\varphi_{R} dx\right)^{\frac{1}{p}}\left(\int_{\Theta_{1}}(\Phi\circ\mu)^{-\frac{1}{p-1}}dy\right)^{\frac{p-1}{p}}\\&
\stackrel{p=\frac{Q}{Q-2}}=C\left(\int_{\overline{\Omega}_R}|u|^{p}\varphi_{R} dx\right)^{\frac{1}{p}}
\left(\int_{\Theta_{1}}(\Phi\circ\mu)^{-\frac{1}{p-1}}dy\right)^{\frac{p-1}{p}}\\&
\leq C\left(\int_{\overline{\Omega}_{R}}|u|^{p}\varphi_{R} dx\right)^{\frac{1}{p}}.
\end{split}
\end{equation*}
Finally, we get 
\begin{equation*}
\begin{split}
\int_{\H}|u|^p\varphi_{R} dx&\leq C\left(\int_{\overline{\Omega}_{R}}|u|^{p}\varphi_{R} dx\right)^{\frac{1}{p}}.
\end{split}
\end{equation*}
By letting $R\rightarrow \infty$ and using \eqref{00}, we get
$$
\int_{\H}|u|^p dx=\lim_{R\rightarrow \infty}\int_{\H}|u|^p \varphi_{R}dx\leq C\lim_{R\rightarrow \infty}\left(\int_{\overline{\Omega}_{R}}|u|^{p}\varphi_{R} dx\right)^{\frac{1}{p}}=0,
$$
that is, $u=0,$ ending the proof.
\end{proof}

\subsection{Elliptic type systems}

In this subsection we consider the following system of Kirchhoff type elliptic equations
\begin{equation}\label{ellsys}
\begin{cases}
-M_{1}\left(\int_{\H}|\na u|^{2}dx,\int_{\H}|\L u|^{2}dx\right)\L u = |v|^{q},   \,\, x \in \H,\,\, q > 1,\\
-M_{2}\left(\int_{\H}|\na v|^{2}dx,\int_{\H}|\L v|^{2}dx\right)\L v = |u|^{p},   \,\, x\in \H,\,\, p > 1,\\

\end{cases}
\end{equation}
where $u,v\in C^{2}(\H)$ and $M_{1},M_{2}:\mathbb{R}_{+}\times\mathbb{R}_{+}\rightarrow\mathbb{R}$ are  bounded functions such that
\begin{equation}\label{sysell1}
0<M_{1}(\cdot,\cdot)\leq C_{1},
\end{equation}
and
\begin{equation}\label{sysell2}
0<M_{2}(\cdot,\cdot)\leq C_{2}.
\end{equation}

Now, we formulate a definition of weak solutions to the system \eqref{ellsys}.
\begin{defn}
We say that a pair of functions $(u, v)\in S^{2}_{2}(\H)\cap L^{p}(\H)\times S^{2}_{2}(\H)\cap L^{q}(\H)$ with $p,q>1$ is a  weak solution of the system of equations \eqref{ellsys} on the group $\H$,  if the following identities 
\begin{equation}
\int_{\H}|v|^q\varphi dx
=  - \int_{\H} M_{1}\left(\int_{\H}|\na u|^{2}dx,\int_{\H}|\L u|^{2}dx\right) u\L \varphi dx,
\end{equation}
and
\begin{equation}
\int_{\H}|u|^p\psi dx
=  - \int_{\H} M_{2}\left(\int_{\H}|\na v|^{2}dx,\int_{\H}|\L v|^{2}dx\right) v\L \psi dx,
\end{equation}
hold true, for any test functions $0\leq\psi,\varphi\in C_{0}^{2}(\H)$.
\end{defn}

In the following statement, we give a nonexistence result for the system \eqref{ellsys}.
\begin{thm}
Let $p,q>1$. Suppose $M_{1}$ and $M_{2}$ satisfy \eqref{sysell1} and \eqref{sysell2}, respectively. Assume that  the following inequality
\begin{equation}
Q\left(1-\frac{1}{pq}\right)\leq 2\max\left\{1+\frac{1}{q},1+\frac{1}{p}\right\},
\end{equation}
holds with $Q=2n+2$ the homogeneous dimension of $\H$. Then the system of elliptic type equations \eqref{ellsys} does not have global non-trivial weak solution.
\end{thm}

\begin{proof}
In analogy with the single equation case, one can  show that
\begin{equation}
\int_{\H}|v|^{q}\varphi dx\leq C_{0}\int_{\H}|u||\L\varphi| dx,\,\,\,\,\,\int_{\H}|u|^{p}\psi dx\leq C_{1}\int_{\H}|v||\L\psi| dx.
\end{equation}
Now, we choose the following test functions
\begin{equation*}
\varphi_R (x) = \Phi\left(\frac{|\xi|^{4}+|\tilde{\xi}|^{4}+\tau^{2}}{R^{4}}\right), \,\,\,x=(\xi,\tilde{\xi},\tau),\,\, R>0,
\end{equation*}
and 
\begin{equation*}
\psi_R (x) = \Psi \left(\frac{|\xi|^{4}+|\tilde{\xi}|^{4}+\tau^{2}}{R^{4}}\right), \,\,\,x=(\xi,\tilde{\xi},\tau),\,\, R>0,
\end{equation*}
where $\Phi,\Psi\in C^{\infty}_{0}(\mathbb{R}_{+})$ have the following property
\[
\Phi(r),\Psi(r) =
\begin{cases}
1, & \text{if $ 0\leq r \leq 1 $,} \\
\searrow, & \text{if $1< r\leq 2$,}\\
0, & \text{if $r>  2$,}
\end{cases}
\]
with $\text{supp}\,\Psi=\text{supp}\,\Phi=[0,2]$. We note that $\text{supp}(\varphi_{R})$, $\text{supp}(\psi_{R})$,  $\text{supp}(\L\varphi_{R})$ and $\text{supp}(\L\psi_{R})$  are subsets of $\Omega_{1}$, where
\begin{equation*}
\Omega_{1}:=\{x=(\xi,\tilde{\xi},\tau)\in\H: |\xi|^{4}+|\tilde{\xi}|^{4}+\tau^{2}\leq2R^{4}\}.
\end{equation*}
Then for the first inequality, by using H\"{o}lder's inequality, one obtains
\begin{equation}\label{1ocenellsys}
\begin{split}
\int_{\H}|v|^{q}\varphi_{R} dx&\leq C_{0}\int_{\H}|u||\L\varphi_{R}| dx\\&
=C_{0}\int_{\Omega_{1}}|u||\L\varphi_{R}| dx\\&
=C_{0}\int_{\Omega_{1}}|u|\psi_{R}^{\frac{1}{p}}\psi_{R}^{-\frac{1}{p}}|\L\varphi_{R}| dx
\\&\leq C_{0}\left(\int_{\Omega_{1}}|u|^{p}\psi_{R} dx\right)^{\frac{1}{p}} \left(\int_{\Omega_{1}}\psi_{R}^{-\frac{1}{p-1}}|\L\varphi_{R}|^{\frac{p}{p-1}} dx\right)^{\frac{p-1}{p}}\\&
= C_{0}\left(\int_{\H}|u|^{p}\psi_{R} dx\right)^{\frac{1}{p}} \left(\int_{\Omega_{1}}\psi_{R}^{-\frac{1}{p-1}}|\L\varphi_{R}|^{\frac{p}{p-1}} dx\right)^{\frac{p-1}{p}}.
\end{split}
\end{equation}
Similarly, we have
\begin{equation}\label{2ocenellsys}
\int_{\H}|u|^{p}\psi_{R} dx\leq C_{1}\left(\int_{\H}|v|^{q}\varphi_{R} dx\right)^{\frac{1}{q}} \left(\int_{\Omega_{1}}\varphi_{R}^{-\frac{1}{q-1}}|\L\psi_{R}|^{\frac{q}{q-1}} dx\right)^{\frac{q-1}{q}}.
\end{equation}
By combining \eqref{1ocenellsys} and \eqref{2ocenellsys}, one finds
\begin{equation}\label{ellsys1ocen}
\begin{split}
&\int_{\H}|v|^{q}\varphi_{R} dx\leq  C_{0}\left(\int_{\H}|u|^{p}\psi_{R} dx\right)^{\frac{1}{p}} \left(\int_{\Omega_{1}}\psi_{R}^{-\frac{1}{p-1}}|\L\varphi_{R}|^{\frac{p}{p-1}} dx\right)^{\frac{p-1}{p}}
\\&\leq C_{0}C^{\frac{1}{p}}_{1}\left(\int_{\H}|v|^{q}\varphi_{R} dx\right)^{\frac{1}{pq}}\left(\int_{\Omega_{1}}\varphi_{R}^{-\frac{1}{q-1}}|\L\psi_{R}|^{\frac{q}{q-1}} dx\right)^{\frac{q-1}{pq}}\\&
\times\left(\int_{\Omega_{1}}\psi_{R}^{-\frac{1}{p-1}}|\L\varphi_{R}|^{\frac{p}{p-1}} dx\right)^{\frac{p-1}{p}}.
\end{split}
\end{equation}
Finally, we obtain
\begin{equation}
\label{pqsys1}
\begin{split}
\left(\int_{\H}|v|^{q}\varphi_{R} dx\right)^{1-\frac{1}{pq}}
&\leq C_{0}C^{\frac{1}{p}}_{1}\left(\int_{\Omega_{1}}\varphi_{R}^{-\frac{1}{q-1}}|\L\psi_{R}|^{\frac{q}{q-1}} dx\right)^{\frac{q-1}{pq}}\\&
\times\left(\int_{\Omega_{1}}\psi_{R}^{-\frac{1}{p-1}}|\L\varphi_{R}|^{\frac{p}{p-1}} dx\right)^{\frac{p-1}{p}}.
\end{split}
\end{equation}
Also, we have
\begin{equation}
\label{pqsys2}
\begin{split}
\left(\int_{\H}|u|^{p}\psi_{R} dx\right)^{1-\frac{1}{pq}}&\leq C_{1}C^{\frac{1}{q}}_{0}\left(\int_{\Omega_{1}}\varphi_{R}^{-\frac{1}{q-1}}|\L\psi_{R}|^{\frac{q}{q-1}} dx\right)^{\frac{q-1}{q}}\\&
\times\left(\int_{\Omega_{1}}\psi_{R}^{-\frac{1}{p-1}}|\L\varphi_{R}|^{\frac{p}{p-1}} dx\right)^{\frac{p-1}{pq}}.
\end{split}
\end{equation}

Taking into account \eqref{poxxell}, one obtains
\begin{equation}
|\L \psi_{R}|\leq CR^{-2},\,\,\,|\L \varphi_{R}|\leq CR^{-2}.
\end{equation}
Similarly to equation case, let us denote 
$$\Theta:=\{y=(\overline{\xi},\widehat{\xi},\tilde{\tau})\in \H:|\overline{\xi}|^{4}+|\widehat{\xi}|^{4}+\tilde{\tau}^{2}\leq 2\},$$
and 
$$\mu:=|\overline{\xi}|^{4}+|\widehat{\xi}|^{4}+\tilde{\tau}^{2}.$$
By changing variables $R\overline{\xi}=\xi$, $R\widehat{\xi}=\tilde{\xi}$ and $R^{2}\tilde{\tau}=\tau$ in \eqref{pqsys1}, we get
\begin{equation}\label{pqsys3}
\begin{split}
&\left(\int_{\H}|v|^{q}\varphi_{R} dx\right)^{1-\frac{1}{pq}}\\&\leq C_{0}C^{\frac{1}{p}}_{1}\left(\int_{\Omega_{1}}\varphi_{R}^{-\frac{1}{q-1}}|\L\psi_{R}|^{\frac{q}{q-1}} dx\right)^{\frac{q-1}{pq}}\left(\int_{\Omega_{1}}\psi_{R}^{-\frac{1}{p-1}}|\L\varphi_{R}|^{\frac{p}{p-1}} dx\right)^{\frac{p-1}{p}}
\\&\leq C_{0}C^{\frac{1}{p}}_{1}R^{-\frac{2}{p}+\frac{Q(q-1)}{pq}+\frac{Q(p-1)}{p}-2}\left(\int_{\Theta}(\Phi\circ\mu)^{-\frac{1}{q-1}} dy\right)^{\frac{q-1}{pq}}\left(\int_{\Theta}(\Psi\circ\mu)^{-\frac{1}{p-1}} dy\right)^{\frac{p-1}{p}}
\\&=C_{0}C^{\frac{1}{p}}_{1}R^{l_{1}}\left(\int_{\Theta}(\Phi\circ\mu)^{-\frac{1}{q-1}} dy\right)^{\frac{q-1}{pq}}\left(\int_{\Theta}(\Psi\circ\mu)^{-\frac{1}{p-1}} dy\right)^{\frac{p-1}{p}}\\&
\leq CR^{l_{1}},
\end{split}
\end{equation}
where $l_{1}=-\frac{2}{p}+\frac{Q(q-1)}{pq}+\frac{Q(p-1)}{p}-2$.

Also, we have
\begin{equation*}\label{pqsys4}
\begin{split}
&\left(\int_{\H}|u|^{p}\psi_{R} dx\right)^{1-\frac{1}{pq}} \\&
\leq C_{1}C_{0}^{\frac{1}{q}}\left(\int_{\Omega_{1}}\varphi_{R}^{-\frac{1}{q-1}}|\L\psi_{R}|^{\frac{q}{q-1}} dx\right)^{\frac{q-1}{q}}\left(\int_{\Omega_{1}}\psi_{R}^{-\frac{1}{p-1}}|\L\varphi_{R}|^{\frac{p}{p-1}} dx\right)^{\frac{p-1}{pq}}\\&
=C_{1}C_{0}^{\frac{1}{q}}R^{l_{2}}\left(\int_{\Theta}(\Phi\circ\mu)^{-\frac{1}{q-1}} dy\right)^{\frac{q-1}{q}}\left(\int_{\Theta}(\Psi\circ\mu)^{-\frac{1}{p-1}} dy\right)^{\frac{p-1}{pq}}
\\&\leq CR^{l_{2}},
\end{split}
\end{equation*}
where $l_{2}=-\frac{2}{q}+\frac{Q(p-1)}{pq}+\frac{Q(q-1)}{q}-2$.

Simple calculations show that
\begin{equation}
\begin{split}
l_{1}&=-\frac{2}{p}+\frac{Q(q-1)}{pq}+\frac{Q(p-1)}{p}-2
\\&=-\frac{2}{p}-2+Q\left(\frac{1}{p}-\frac{1}{pq}+1-\frac{1}{p}\right)
\\&=-\frac{2}{p}-2+Q\left(1-\frac{1}{pq}\right)<0,
\end{split}
\end{equation}
provided that
$$
Q<\frac{2+\frac{2}{p}}{1-\frac{1}{pq}}.
$$

Then letting $R\rightarrow\infty$, we obtain
$$
\left(\int_{\H}|v|^{q}\varphi_{R} dx\right)^{1-\frac{1}{pq}}\leq0,
$$
showing that $\int_{\H}|v|^{q} dx=0$.

Similar calculations works in the case $l_{2}<0$, i.e.,
$$
Q\left(1-\frac{1}{pq}\right)< 2\left(1+\frac{1}{q}\right).
$$
Finally, letting $R\rightarrow\infty$, we get
$$
\int_{\H}|u|^{p} dx=0.
$$

Let us consider the case $l_{1}=0$ (or $l_{2}=0$).
By denoting
\begin{equation*}
\overline{\Omega}_{R}:=\{x=(\xi,\tilde{\xi},\tau)\in\H:R^{4}\leq |\xi|^{4}+|\tilde{\xi}|^{4}+\tau^{2}\leq2R^{4}\},
\end{equation*}
we have 
\begin{equation}
  \lim\limits_{R\rightarrow \infty}  \int_{\overline{\Omega}_{R}}|v|^{q}\varphi_{R} dx=0,
\end{equation}
and 
\begin{equation}
  \lim\limits_{R\rightarrow \infty}  \int_{\overline{\Omega}_{R}}|u|^{p}\psi_{R} dx=0.
\end{equation}
Let us denote by 
$$\Theta_{1}:=\{y=(\overline{\xi},\widehat{\xi},\tilde{\tau})\in \H:1\leq |\overline{\xi}|^{4}+|\widehat{\xi}|^{4}+\tilde{\tau}_{1}^{2}\leq 2\}.$$
From \eqref{ellsys1ocen}, \eqref{pqsys2} and changing variables $R\overline{\xi}=\xi$, $R\widehat{\xi}=\tilde{\xi}$ and $R^{2}\tilde{\tau}=\tau$, we establish
\begin{equation}\label{pqsysel3}
\begin{split}
\int_{\H}|v|^{q}\varphi_{R} dx&\leq C_{0}\int_{\H}|u||\L\varphi_{R}| dx\\&
=C_{0}\int_{\overline{\Omega}_{R}}|u||\L\varphi_{R}| dx\\&
\leq C_{0}\left(\int_{\overline{\Omega}_{R}}|v|^{q}\varphi_{R} dx\right)^{\frac{1}{pq}}\left(\int_{\overline{\Omega}_{R}}\varphi_{R}^{-\frac{1}{q-1}}|\L\psi_{R}|^{\frac{q}{q-1}} dx\right)^{\frac{q-1}{pq}}\\&
\times\left(\int_{\overline{\Omega}_{R}}\psi_{R}^{-\frac{1}{p-1}}|\L\varphi_{R}|^{\frac{p}{p-1}} dx\right)^{\frac{p-1}{p}}\\&
\leq C_{0}R^{l_{1}}\left(\int_{\overline{\Omega}_{R}}|v|^{q}\varphi_{R} dx\right)^{\frac{1}{pq}}
\left(\int_{\Theta_{1}}[\Phi\circ\rho]^{-\frac{1}{q-1}} dy\right)^{\frac{q-1}{pq}}\\&
\times\left(\int_{\Theta_{1}}[\Psi\circ \rho]^{-\frac{1}{p-1}} dy\right)^{\frac{p-1}{p}}\\&
\stackrel{l_{1}=0}\leq C\left(\int_{\overline{\Omega}_{R}}|v|^{q}\varphi_{R} dx\right)^{\frac{1}{pq}},
\end{split}
\end{equation}
and  
\begin{equation}
\begin{split}
\int_{\H}|u|^{p}\psi_{R} dx&\leq C \left(\int_{\overline{\Omega}_{R}}|u|^{p}\psi_{R} dx\right)^{\frac{1}{pq}}\left(\int_{\overline{\Omega}_{R}}\varphi_{R}^{-\frac{1}{q-1}}|\L\psi_{R}|^{\frac{q}{q-1}} dx\right)^{\frac{q-1}{q}}\\&
\times\left(\int_{\overline{\Omega}_{R}}\psi_{R}^{-\frac{1}{p-1}}|\L\varphi_{R}|^{\frac{p}{p-1}} dx\right)^{\frac{p-1}{pq}}\\&
\leq C \left(\int_{\overline{\Omega}_{R}}|u|^{p}\psi_{R} dx\right)^{\frac{1}{pq}}.
\end{split}
\end{equation}
Hence, we obtain
\begin{equation}
    \int_{\H}|v|^{q} dx=\lim_{R\rightarrow \infty}\int_{\H}|v|^{q}\varphi_{R} dx\leq C\lim\limits_{R\rightarrow \infty}\left(\int_{\overline{\Omega}_{R}}|v|^{q}\varphi_{R} dx\right)^{\frac{1}{pq}}=0,
\end{equation}
and 
\begin{equation}
    \int_{\H}|u|^{p}dx=\lim_{R\rightarrow \infty}\int_{\H}|u|^{p}\psi_{R} dx\leq C\lim\limits_{R\rightarrow \infty}\left(\int_{\overline{\Omega}_{R}}|u|^{p}\psi_{R} dx\right)^{\frac{1}{pq}}=0,
\end{equation}
ending the proof of the theorem.
\end{proof}

\begin{cor}
Let $p=q>1$, $u=v$ and $M_{1}=M_{2}$. Then we arrive at the result of Theorem \ref{ellthm}.
\end{cor}
\begin{proof}
It is easy to see that
$$
Q\left(\frac{p^{2}-1}{p^{2}}\right)=Q\left(1-\frac{1}{p^{2}}\right)\leq 2\left(1+\frac{1}{p}\right)=2\left(\frac{p+1}{p}\right),
$$
showing
$$
p\leq \frac{Q}{Q-2}.
$$
The proof is finished.
\end{proof}

\section{Kirchhoff type parabolic equations and systems}
In this section we prove the Liouville type theorems for the Kirchhoff type parabolic equations and systems on the Heisenberg group $\H$.
\subsection{Parabolic type  equations}
Let us consider the following Kirchhoff type heat equation on $\H$, that is,
\begin{equation}\label{heat}
\begin{cases}
u_{t} -M\left(t,\int_{\H}|\na u|^{2}dx,\int_{\H}|\L u|^{2}dx\right) \L u= |u|^{p},\\
u(x, 0) = u_0(x)\geq0, \,\,\,x\in\H,
\end{cases}
\end{equation}
where $(x, t) \in\ \H \times (0, T),\,T>0,\, p > 1$, $M:\mathbb{R}_{+}\times\mathbb{R}_{+}\times\mathbb{R}_{+}\rightarrow\mathbb{R}$ is satisfying the following condition
\begin{equation}\label{Mh}
0<M(t,\cdot,\cdot)\leq C_{1}t^{\beta},\,\,\forall t>0,\,\,0\leq\beta<1.
\end{equation}

Let us formulate the following definition of the weak solution of the equation \eqref{heat}.
\begin{defn}
We say that $u$ is a weak solution to \eqref{heat} in $\Q=\H\times(0,T)$ with initial data
$0\leq u_{0} \in L^1(\H), $ if $ u\in C^{1}(0,T;S^{2}_{2}(\H))\cap L^p_{loc}(\Q) $ and satisfies\\
\begin{multline}
\int_{\Q}|u|^p\varphi dx dt + \int_{\H} u_{0}(x)\varphi(x,0)dx =-\int_{\Q} u \varphi_{t} dx dt \\
 - \int_{\Q} M\left(t,\int_{\H}|\na u|^{2}dx,\int_{\H}|\L u|^{2}dx\right) u\L \varphi dx dt,
\end{multline}
for any test function $0\leq\varphi\in C^{2,1}_{x,t}(\Q)$.
\end{defn}

Now we give one of the main results of this section.
\begin{thm}\label{heatthm}
Suppose that $M(\cdot,\cdot,\cdot)$ satisfies the condition \eqref{Mh}. If
\begin{equation}
1<p< p_{c}=1+\frac{2-2\beta}{Q+2\beta},
\end{equation}
with $Q=2n+2$ the homogeneous dimension of $\H$, then the problem \eqref{heat} admits no global  weak non-trivial  solution.
\end{thm}
\begin{proof}
Let us choose
\begin{equation}
\label{testfunch}
\varphi_R (x, t) = \Phi \left(\frac{|
\xi|^{4}+|\tilde{\xi}|^{4}+\tau^{2}+t^{2}}{R^{4}}\right), \,\,\,\,x=(\xi,\tilde{\xi},\tau)\in\H,\,\,t>0,\,\,R>0,
\end{equation}
with $\Phi\in C^{\infty}_{0}(\mathbb{R}_{+})$ and the following property
\[
\Phi(r) =
\begin{cases}
1, & \text{if $ 0\leq r \leq 1 $,} \\
\searrow, & \text{if $1< r\leq 2$,}\\
0, & \text{if $r>  2$.}
\end{cases}
\]
We note that  $\text{supp}(\varphi_{R})$ and $\text{supp}(\L\varphi_{R})$ are subsets of $\Omega_{1}$, where
\begin{equation}
\Omega_{1}:=\{(x,t)=(\xi,\tilde{\xi},\tau,t)\in\H\times
\mathbb{R}_{+}: |
\xi|^{4}+|\tilde{\xi}|^{4}+\tau^{2}+t^{2}\leq2R^{4}\}.
\end{equation}
We seek the test function $\varphi_{R}$ with the following properties
\begin{equation*}
\label{test1}
\int_{\Omega_{1}}|(\varphi_{R}(x,t))_{t}|^{p'}|\varphi_{R}(x,t)|^{-\frac{p'}{p}}dxdt<\infty,
\end{equation*}
and
\begin{equation*}
\label{test2}
\int_{\Omega_{1}}t^{\beta p'}|\L \varphi_{R}(x,t)|^{p'}|\varphi_{R}(x,t)|^{-\frac{p'}{p}}dxdt<\infty.
\end{equation*}

Consider the following integral in the domain $\Omega$,
\begin{equation*}
\begin{split}
&\int_{\Omega}\varphi_{R}(x,t)|u(x,t)|^{p}dxdt\leq \int_{\Omega}\varphi_{R}(x,t)|u(x,t)|^{p}dxdt+\int_{\H}u_{0}(x)\varphi_{R}(x,0)dxdt
\\&=-\int_{\Omega}M\left(t,\int_{\H}|\na u|^{2}dx,\int_{\H}|\L u|^{2}dx\right)u(x,t)\L\varphi_{R}(x,t)dxdt
\\&-\int_{\Q}u(x,t)(\varphi_{R}(x,t))_{t}dxdt
 \\&
 \leq C_{1}\int_{\Omega}t^{\beta}|u(x,t)||\L\varphi_{R}(x,t)|dxdt+\int_{\Q}|u(x,t)||(\varphi_{R}(x,t))_{t}|dxdt.
\end{split}
\end{equation*}
Then we have
\begin{multline*}
\int_{\Omega}\varphi_{R}(x,t)|u(x,t)|^{p}dxdt\leq C_{1}\int_{\Omega}t^{\beta}|u(x,t)||\L\varphi_{R}(x,t)|dxdt+\int_{\Omega}|u(x,t)||(\varphi_{R}(x,t))_{t}|dxdt.
\end{multline*}
By using $\varepsilon$-Young's inequality  in the first integral of the right hand side, one obtains
\begin{equation}
\label{hol3}
\begin{split}
&\int_{\Omega}t^{\beta}|u(x,t)||\L\varphi_{R}(x,t)|dxdt
\\&=\int_{\Omega_{1}}(\varphi_{R}(x,t))^{\frac{1}{p}}|u(x,t)|t^{\beta}|\L\varphi(x,t)|(\varphi_{R}(x,t))^{-\frac{1}{p}}dxdt
\\&\leq\varepsilon \int_{\Omega}\varphi_{R}(x,t) |u(x,t)|^{p}dxdt
+C(\varepsilon) \int_{\Omega_{1}}|t|^{\beta p'}|\L \varphi_{R}(x,t)|^{p'}(\varphi_{R}(x,t))^{-\frac{p'}{p}}dxdt,
\end{split}
\end{equation}
and
\begin{multline}\label{hol2}
\int_{\Omega}|u(x,t)||(\varphi_{R}(x,t)))_{t}|dxdt\leq \varepsilon_{1}\int_{\Omega}|u(x,t)|^{p}\varphi_{R}(x,t)dxdt\\
+C(\varepsilon_{1})\int_{\Omega_{1}}|(\varphi_{R}(x,t))_{t}|^{p'}|\varphi_{R}(x,t)|^{-\frac{p'}{p}}dxdt.
\end{multline}
Then from these facts and choosing $C_{1}\varepsilon+\varepsilon_{1}<1$, we have
\begin{multline}\label{obsh}
\int_{\Omega}|u(x,t)|^{p}\varphi_{R}(x,t) dxdt\\
\leq C\left(\int_{\Omega_{1}}(t^{\beta p'}|\L \varphi_{R}(x,t)|^{p'}+|(\varphi_{R}(x,t)))_{t}|^{p'})(\varphi_{R}(x,t))^{-\frac{p'}{p}}dxdt\right).
\end{multline}

We note that from \cite{PohVer} we have
\begin{equation*}
|\L \varphi_{R}|\leq CR^{-2},
\end{equation*}
and
\begin{equation*}
\left|\frac{\partial \varphi_{R}}{\partial t}\right|\leq CR^{-2}.
\end{equation*}
Let us denote 
$$\Theta:=\{(y,t_{1})=(\overline{\xi},\widehat{\xi},\tilde{\tau},t_{1})\in \H\times\mathbb{R}_{+}:|\overline{\xi}|^{4}+|\widehat{\xi}|^{4}+\tilde{\tau}^{2}+t_{1}^{2}\leq 2\},$$
and 
$$\mu:=|\overline{\xi}|^{4}+|\widehat{\xi}|^{4}+\tilde{\tau}^{2}+t_{1}^{2}.$$
Since the homogeneous dimension of $\H$ is equal to $Q,$ and by changing variables $R^{2}\overline{t}=t$, $R\overline{\xi}=\xi$, $R\widehat{\xi}=\tilde{\xi}$ and $R^{2}\tilde{\tau}=\tau$ in $\Omega_{1}$, we calculate
\begin{equation}\label{ocen1h}
\begin{split}
\int_{\Omega_{1}}|\varphi_{R}(x,t)|^{-\frac{1}{p-1}}\left|\frac{\partial\varphi_{R}(x,t)}{\partial t}\right|^{\frac{p}{p-1}}dxdt
&=R^{Q+2-\frac{2p}{p-1}}\int_{\Theta}|(\Phi\circ\mu)|^{-\frac{p}{p-1}}\left|\frac{\partial(\Phi\circ\mu)}{\partial t}\right|^{\frac{p}{p-1}}dydt_{1}\\&
\leq C R^{Q+2-\frac{2p}{p-1}},
\end{split}
\end{equation}
and
\begin{equation}\label{ocen3h}
\begin{split}
\int_{\Omega_{1}}t^{\beta p'}|\varphi_{R}(x,t)|^{-\frac{1}{p-1}}&\left|\L \varphi_{R}(x,t)\right|^{\frac{p}{p-1}}dxdt \\&
=R^{\frac{2\beta p}{p-1}-\frac{2p}{p-1}+Q+2}\int_{\Theta}t_{1}^{\beta p'}|(\Phi\circ\mu)|^{-\frac{p}{p-1}}\left|\L (\Phi\circ\mu)\right|^{\frac{p}{p-1}}dydt_{1}\\&
\leq C R^{\frac{2\beta p}{p-1}-\frac{2p}{p-1}+Q+2}.
\end{split}
\end{equation}

By using \eqref{ocen1h}--\eqref{ocen3h} in \eqref{obsh}, one obtains
\begin{equation}\label{ocenk4h}
\begin{split}
\int_{\Omega}|u|^p \varphi_{R} dxdt &\leq C\left(R^{Q+2-\frac{2p}{p-1}}+R^{Q+2-\frac{2p}{p-1}+\frac{2\beta p}{p-1}}\right)
\\&
\leq CR^{Q+2-\frac{2p}{p-1}+\frac{2\beta p}{p-1}}.
\end{split}
\end{equation}
Here, we choose
$$Q+2-\frac{2p}{p-1}+\frac{2\beta p}{p-1}<0,$$
that is,
$$p<1+\frac{2-2\beta}{Q+2\beta}.$$
Letting $R\rightarrow \infty$ with $p<1+\frac{2-2\beta}{Q+2\beta}$, one has
\begin{equation}
\int_{\Omega}|u(x,t)|^{p}dxdt\leq0,
\end{equation}
arriving at a contradiction.
\end{proof}

\subsection{Parabolic type systems}

Now, we will show a Liouville type result for the following system of Kirchhoff type parabolic equations in $\Q:= \H \times (0, T),$ $T>0$
\begin{equation}\label{heatsys}\small
\begin{cases}
u_{t}-M_{1}\left(t,\int\limits_{\H} |\na u|^{2}dx, \int\limits_{\H}|\na \omega|^{2}dx, \int\limits_{\H}|\L u|^{2}dx, \int\limits_{\H}|\L \omega|^{2}dx\right)\L u = |\omega|^{p},\\
\omega_{t} - M_{2}\left(t,\int\limits_{\H} |\na u|^{2}dx, \int\limits_{\H}|\na \omega|^{2}dx, \int\limits_{\H}|\L u|^{2}dx, \int\limits_{\H}|\L \omega|^{2}dx\right)\L \omega = |u|^{q},\\
u(x, 0) = u_0(x)\geq0, \,\,\,x\in\H,\\
\omega(x,0)=\omega_{0}(x)\geq0, \,\,\,x\in\H,
\end{cases}
\end{equation}
where $p,q>1$ and $M_{1},M_{2}:\mathbb{R}_{+}\times\mathbb{R}_{+}\times\mathbb{R}_{+}\rightarrow\mathbb{R}$ are  bounded functions such that
\begin{equation}\label{hsysus1}
0<M_{1}(t,\cdot,\cdot, \cdot, \cdot)\leq C_{1}t^{\beta_{1}},\,\,\forall t\in(0,\infty),\,\,0\leq\beta_{1}<1,
\end{equation}
and
\begin{equation}\label{hsysus2}
0<M_{2}(t,\cdot,\cdot, \cdot, \cdot)\leq C_{2}t^{\beta_{2}},\,\,\forall t\in(0,\infty),\,\,0\leq\beta_{2}<1.
\end{equation}

Let us give a definition of the weak solution to \eqref{heatsys} as follows.
\begin{defn}
We say that the pair $u\in  C^{1}(0,T;S^{2}_{2}(\H))\cap L^p(\Q)$ and $ \omega\in C^{1}(0,T;S^{2}_{2}(\H)) \cap L^q(\Q)$  with $p,q>1$ is a weak solution of the system \eqref{heatsys} on $\Q=\H\times(0,T)$ with the Cauchy data $(u_0, \omega_0) \in  L^1_{loc}(\H) \times L^1_{loc}(\H)$ with $u_{0},\omega_{0}\geq0$ on $\H$, if the following identities
\begin{align*}
 &\int_{\Q}|\omega|^{p} \varphi dxdt +\int_{\H} u_0(x) \varphi(x,0) dx=-\int_{\Q} u \varphi_{t} dxdt \\&
 -\int_{\Q} M_{1}\left(t,\int\limits_{\H} |\na u|^{2}dx, \int\limits_{\H}|\na \omega|^{2}dx, \int\limits_{\H}|\L u|^{2}dx, \int\limits_{\H}|\L \omega|^{2}dx\right)u \L\varphi dxdt,
\end{align*}
and
\begin{align*}
 &\int_{\Q}|u|^q \psi dxdt +\int_{\H} \omega_0(x)\psi(x,0)dx
=-\int_{\Q} \omega \psi_{t} dxdt \\&
- \int_{\Q} M_{2}\left(t,\int\limits_{\H} |\na u|^{2}dx, \int\limits_{\H}|\na \omega|^{2}dx, \int\limits_{\H}|\L u|^{2}dx, \int\limits_{\H}|\L \omega|^{2}dx\right)w \L\psi dxdt,
\end{align*}
hold for any test functions $0\leq\psi,\varphi\in C^{2,1}_{x,t}(\Q)$.
\end{defn}

\begin{thm}\label{thmheatsysblow}
Let $p,q>1$ and $Q<\max\{A_{1},A_{2}\},$ where $Q$ is the homogeneous dimension of $\H$,
\begin{equation}
A_{1}=\frac{\frac{1}{q}-\beta_{1}-\frac{\beta_{2}}{q}+\frac{1}{pq}}{\frac{1}{2qp'}+\frac{1}{2q'}},
\end{equation}
and
\begin{equation}
A_{2}=\frac{\frac{1}{p}-\beta_{2}-\frac{\beta_{1}}{p}+\frac{1}{pq}}{\frac{1}{2q'p}+\frac{1}{2p'}},
\end{equation}
where $\beta_{1},\beta_{2}\in[0,1).$
Then the system \eqref{heatsys} does not admit a nontrivial  weak solution.
\end{thm}
\begin{proof}
Firstly, let us define the test functions
\begin{equation}
\varphi_{R}=\Phi\left(\frac{|
\xi|^{4}+|\tilde{\xi}|^{4}+\tau^{2}+t^{2}}{R^{4}}\right), \,\,\,\,x=(\xi,\tilde{\xi},\tau)\in\H,\,\,t>0,\,\,\,R>0,
\end{equation}
and 
\begin{equation}
\psi_{R}=\Psi\left(\frac{|
\xi|^{4}+|\tilde{\xi}|^{4}+\tau^{2}+t^{2}}{R^{4}}\right), \,\,\,\,x=(\xi,\tilde{\xi},\tau)\in\H,\,\,t>0,\,\,\,R>0,
\end{equation}
with $\Psi,\Phi\in C^{\infty}_{0}(\mathbb{R}_{+})$ and the following property
\[
\Psi(r),\Phi(r) =
\begin{cases}
1, & \text{if $ 0\leq r \leq 1 $,} \\
\searrow, & \text{if $1< r\leq 2$,}\\
0, & \text{if $r>  2$,}
\end{cases}
\]
with $\text{supp} \,\Psi=\text{supp} \,\Phi=[0,2].$ We note that  $\text{supp}(\psi_{R})$, $\text{supp}(\varphi_{R})$, $\text{supp}(\L\psi_{R})$ and $\text{supp}(\L\varphi_{R})$ are subsets of $\Omega_{1}$, where
\begin{equation}
\Omega_{1}:=\{(x,t)=(\xi,\tilde{\xi},\tau,t)\in\H\times
\mathbb{R}_{+}: |
\xi|^{4}+|\tilde{\xi}|^{4}+\tau^{2}+t^{2}\leq2R^{4}\}.
\end{equation}

Assume that there exists a weak solution to \eqref{heatsys}, and let $M_{1}$ and $M_2$ satisfies the conditions \eqref{hsysus1} and \eqref{hsysus2}, then
\begin{equation}\label{weak1}
\begin{split}
&\int_{\Omega} |\omega|^{p}\varphi_{R}dxdt+\int_{\H}u_{0}(x)\varphi_{R}(x,0)dx
\\& \leq\int_{\Omega}|u||(\varphi_{R})_{t}|dxdt+C_{1}\int_{\Omega}t^{\beta_{1}}|u||\L\varphi_{R}|dxdt\\&
=\int_{\Omega_{1}}|u||(\varphi_{R})_{t}|dxdt+C_{1}\int_{\Omega_{1}}t^{\beta_{1}}|u||\L\varphi_{R}|dxdt,
\end{split}
\end{equation}
and
\begin{equation*}\label{weak2}
\begin{split}
&\int_{\Omega}|u|^{q}\psi_{R} dxdt+\int_{\H}\omega_{0}(x)\psi_{R}(x,0)dx
\\& \leq\int_{\Omega}|\omega||(\psi_{R})_{t}|dxdt+C_{2}\int_{\Omega}t^{\beta_{2}}|\omega||\L\psi_{R}|dxdt\\&
=\int_{\Omega_{1}}|\omega||(\psi_{R})_{t}|dxdt+C_{2}\int_{\Omega_{1}}t^{\beta_{2}}|\omega||\L\psi_{R}|dxdt.
\end{split}
\end{equation*}
By using H\"{o}lder's inequality, we have
\begin{equation*}
\begin{split}
    \int_{\Omega_{1}}|u||(\varphi_{R})_{t}| dxdt&=\int_{\Omega_{1}}|u|\psi_{R}^{\frac{1}{q}}|(\varphi_{R})_{t}| \psi_{R}^{-\frac{1}{q}}dxdt\\&
\leq\left(\int_{\Omega_{1}}|u|^{q}\psi_{R} dxdt\right)^{\frac{1}{q}}\left(\int_{\Omega_{1}}|(\varphi_{R})_{t}|^{q'}\psi_{R}^{-\frac{q'}{q}}dxdt\right)^{\frac{1}{q'}}\\&
\leq \left(\int_{\Omega}|u|^{q}\psi_{R} dxdt\right)^{\frac{1}{q}}\left(\int_{\Omega_{1}}|(\varphi_{R})_{t}|^{q'}\psi_{R}^{-\frac{q'}{q}}dxdt\right)^{\frac{1}{q'}},
\end{split}
\end{equation*}

and
\begin{equation*}
\int_{\Omega_{1}}t^{\beta_{1}}|u||\L\varphi_{R}| dxdt\leq \left(\int_{\Omega_{1}}|u|^{q} \psi_{R} dxdt\right)^{\frac{1}{q}}\left(\int_{\Omega_{1}} t^{\beta_{1}q'} |\L\varphi_{R}|^{q'}\psi_{R}^{-\frac{q'}{q}}dxdt\right)^{\frac{1}{q'}}.
\end{equation*}
By using \eqref{weak1} and the previous estimates, we obtain
\begin{equation}\label{ocenkasys1}
\begin{split}
\int_{\Omega} |\omega|^{p}\varphi_{R}dxdt&\leq\int_{\Omega} |\omega|^{p}\varphi_{R}dxdt
+\int_{\H}u_{0}(x)\varphi_{R}(x,0)dxdt
\\&\leq\int_{\Omega}|u|(\varphi_{R})_{t}dxdt
+\int_{\Omega}t^{\beta_{1}}|u||\L\varphi_{R} |dxdt
\\&\leq \left(\int_{\Omega}|u|^{q} \psi_{R} dxdt\right)^{\frac{1}{q}}I_{1},
\end{split}
\end{equation}
where
\begin{equation}\label{i123}
I_{1}=\left(\int_{\Omega_{1}}t^{\beta_{1}q'}|\L\varphi_{R}|^{q'}\psi_{R}^{-\frac{q'}{q}}dxdt\right)^{\frac{1}{q'}}
+\left(\int_{\Omega_{1}}|(\varphi_{R})_{t}|^{q'}\psi_{R}^{-\frac{q'}{q}}dxdt\right)^{\frac{1}{q'}}.
\end{equation}
Analogous calculations show that
\begin{equation}\label{ocenkasys2}
\int_{\Omega}|u|^{q}\psi_{R} dxdt\leq\left(\int_{\Omega}|\omega|^{p}\varphi_{R} dxdt\right)^{\frac{1}{p}}I_{2},
\end{equation}
where
\begin{equation*}\label{i2}
I_{2}=\left(\int_{\Omega_{1}}t^{\beta_{2}p'}|\L\psi_{R}|^{p'}\varphi_{R}^{-\frac{p'}{p}}dxdt\right)^{\frac{1}{p'}}
+\left(\int_{\Omega_{1}}|(\psi_{R})_{t}|^{p'}\varphi_{R}^{-\frac{p'}{p}}dxdt\right)^{\frac{1}{p'}}.
\end{equation*}
By using \eqref{ocenkasys1} and \eqref{ocenkasys2}, one obtains
\begin{equation*}
\begin{split}
\int_{\Omega}|\omega|^{p}\varphi_{R}dxdt\leq\left(\int_{\Omega}|u|^{q} \psi_{R} dxdt\right)^{\frac{1}{q}}I_{2}
\leq\left(\int_{\Omega}|\omega|^{p}\varphi_{R} dxdt\right)^{\frac{1}{pq}}I^{\frac{1}{q}}_{2}I_{1},
\end{split}
\end{equation*}
implying 
\begin{equation*}
\left(\int_{\Omega} |\omega(x,t)|^{p}\varphi_{R}(x,t)dxdt\right)^{1-\frac{1}{pq}}\leq I_{1}I^{\frac{1}{q}}_{2}.
\end{equation*}
Similarly, we have
\begin{equation*}
\left(\int_{\Omega} |u(x,t)|^{q}\psi_{R}(x,t)dxdt\right)^{1-\frac{1}{pq}}\leq I^{\frac{1}{p}}_{1}I_{2}.
\end{equation*}
Let us denote 
$$\Theta:=\{(y,t_{1})=(\overline{\xi},\widehat{\xi},\tilde{\tau},t_{1})\in \H\times\mathbb{R}_{+}:0\leq |\overline{\xi}|^{4}+|\widehat{\xi}|^{4}+\tilde{\tau}^{2}+t_{1}^{2}\leq 2\},$$
and 
$$\mu:=|\overline{\xi}|^{4}+|\widehat{\xi}|^{4}+\tilde{\tau}^{2}+t_{1}^{2}.$$
By using the homogeneous dimension $Q$ of $\H$  and by choosing variables $R^{2}t_{1}=t$, $R\overline{\xi}=\xi$, $R\widehat{\xi}=\tilde{\xi}$ and $R^{2}\tilde{\tau}=\tau$ in the domain $\Omega_{1}$, one calculates
\begin{equation*}
\begin{split}
\int_{\Omega_{1}}|(\psi_{R})_{t}|^{p'}\varphi_{R}^{-\frac{p'}{p}}dxdt&=R^{-2p'+2+Q}\int_{\Theta}|(\Psi\circ\mu)_{t_{1}}|^{p'}(\Phi\circ\mu)^{-\frac{p'}{p}}dy dt_{1}\\&
\leq CR^{-2p'+2+Q},
\end{split}
\end{equation*}
similarly, we get
\begin{equation*}
\int_{\Omega_{1}}|(\varphi_{R})_{t}|^{q'}\psi_{R}^{-\frac{q'}{q}}dxdt\leq CR^{-2q'+2+Q},
\end{equation*}

\begin{equation*}
\int_{\Omega_{1}}t^{\beta_{1}q'}|\L\varphi_{R}|^{q'}\psi_{R}^{-\frac{q'}{q}}dxdt\leq
CR^{-2q'+2+Q+2\beta_{1}q'},
\end{equation*}
and
\begin{equation*}
\int_{\Omega_{1}}t^{\beta_{2}p'}|\L\psi_{R}|^{p'}\varphi_{R}^{-\frac{p'}{p}}dxdt\leq
CR^{-2p'+2+Q+2\beta_{2}p'}.
\end{equation*}

By using the last estimates in \eqref{i123}, we obtain
\begin{equation*}
I_{1}\leq CR^{-2+\frac{2+Q}{q'}+2\beta_{1}}.
\end{equation*}
Analogously, with $R^{2}t_{1}=t$, $R\overline{\xi}=\xi$, $R\widehat{\xi}=\tilde{\xi}$ and $R^{2}\tilde{\tau}=\tau$ in the domain $\Omega_{1}$, we get
\begin{equation}
I_{2}\leq CR^{-2+\frac{2+Q}{p'}+2\beta_{2}}.
\end{equation}
By using these facts, one has
\begin{equation}
\left(\int_{\Omega}|\omega(x,t)|^{p}\varphi_{R}(x,t)dxdt\right)^{1-\frac{1}{pq}}\leq I_{1}I^{\frac{1}{q}}_{2}\\
\leq C R^{\theta_{1}+\frac{\theta_{2}}{q}},
\end{equation}
where
$\theta_{1}=-2+\frac{2+Q}{q'}+2\beta_{1}$ and $\theta_{2}=-2+\frac{2+Q}{p'}+2\beta_{2}.$ 

Now we show
$$
\theta_{1}+\frac{\theta_{2}}{q}<0.
$$
By calculating, we have
$$
-2+\frac{2+Q}{q'}+2\beta_{1}-\frac{2}{q}+\frac{2\beta_{2}}{q}+\frac{2+Q}{p'q}< 0.
$$
Thus, we can require
\begin{equation}
 Q<\frac{\frac{1}{q}-\beta_{1}-\frac{\beta_{2}}{q}+\frac{1}{pq}}{\frac{1}{2qp'}+\frac{1}{2q'}}.
\end{equation}

Repeating the arguments above, we obtain
\begin{equation}
\left(\int_{\Omega} |u(x,t)|^{q}\psi_{R}(x,t)dxdt\right)^{1-\frac{1}{pq}}\leq I^{\frac{1}{p}}_{1}I_{2}\leq C R^{\theta_{2}+\frac{\theta_{1}}{p}},
\end{equation}
and we can require
\begin{equation}
Q<\frac{\frac{1}{p}-\beta_{2}-\frac{\beta_{1}}{p}+\frac{1}{pq}}{\frac{1}{2q'p}+\frac{1}{2p'}}.
\end{equation}
Finally, letting $R\rightarrow\infty$ we arrive at a contradiction.
\end{proof}

\begin{cor}
Let $p=q$, $u=\omega$, $\beta_{1}=\beta_{2}$ and $M_{1}=M_{2}$. Then we get the results of Theorem \ref{heatthm}.
\end{cor}
\begin{proof}
By simple calculations we have
$$
Q<\frac{\frac{1}{p}-\beta_{1}-\frac{\beta_{1}}{p}+\frac{1}{p^{2}}}{\frac{1}{2p'}\left(\frac{1}{p}+1\right)}=\frac{2-2p\beta_{1}}{p-1},
$$
which means
$$
p< 1+\frac{2-2\beta_{1}}{Q+2\beta_{1}},
$$
implying the statement. 
\end{proof}

\subsection{Pseudo-parabolic type equations and systems}

In this subsection we show nonexistence results for pseudo-parabolic equations and systems on the Heisenberg groups. We consider the following Cauchy problem for the Kirchhoff type pseudo-parabolic equation
\begin{equation}
\label{pseudoheat}
\begin{cases}
u_{t} -M\left(t,\int_{\H}|\na u|^{2}dx,\int_{\H}|\L u|^{2}dx\right)\L u- \L u_{t}= |u|^{p},\\
u(x, 0) = u_0(x)\geq0, \,\,\,x\in\H,
\end{cases}
\end{equation}
where $(x, t) \in\Q:= \H \times (0, T),\,T>0,\, p > 1$, $M:\mathbb{R}_{+}\times\mathbb{R}_{+}\times\mathbb{R}_{+}\rightarrow\mathbb{R}$ is a bounded function such that
\begin{equation}\label{pseudoMh}
0<M(\cdot,\cdot,\cdot)\leq C_{0}.
\end{equation}

We give a definition of the weak solution of the equation \eqref{pseudoheat}.
\begin{defn}
We say that $u\in C^{1}(0,T;S^{2}_{2}(\H))\cap L^{p}(\Q)$ with $p>1$ is a weak solution to \eqref{pseudoheat} on $
\Q$ with the initial data
$0\leq u_{0}(x) \in L^1_{loc}(\H), $ if $ u\in L^p_{loc}(\Q) $ and satisfies
\begin{equation}
\label{pseudoweakheat}
\begin{split}
\int_{\Q}|u|^p\varphi dx dt &+ \int_{\H} u_{0}(x)\varphi(x,0)dx =-\int_{\Q} u \varphi_{t} dx dt \\&
 - \int_{\Q} M\left(t,\int_{\H}|\na u|^{2}dx,\int_{\Omega}|\L u|^{2}dx\right) u\L \varphi dx dt\\&
 +\int_{\Q} u \L \varphi_{t} dx dt+\int_{\H}u_{0}(x)\L\varphi(x,0)dx,
\end{split}
\end{equation}
for any test function $0\leq\varphi\in C^{2,1}_{x,t}(\Q)$.
\end{defn}

Let us formulate one of the main results of this section.
\begin{thm}
\label{pseudoheatthm}
Suppose that $M(\cdot,\cdot,\cdot)$ satisfies the condition \eqref{pseudoMh}.
If the rate $p$ is from the following interval
\begin{equation}
1<p\leq p_{c}=1+\frac{2}{Q},
\end{equation}
with the homogeneous dimension $Q$ of $\H$, then the problem \eqref{pseudoheat} admits no global  weak solution.
\end{thm}

\begin{proof}
Let us give a short proof of Theorem \ref{pseudoheatthm}. 
Here, we choose the test functions as follows
\begin{equation}
\varphi_R (x, t) = \Phi \left(\frac{|\xi|^{4}+|\tilde{\xi}|^{4}+\tau^{2}+t^{2}}{R^{4}}\right), \,\,x=(\xi,\tilde{\xi},\tau)\in \H,\,\,t>0,\,\,R>0,
\end{equation}
with $\Phi\in C^{\infty}_{0}(\mathbb{R}_{+})$ satisfying
\[
\Phi(r) =
\begin{cases}
1, & \text{if $ 0\leq r \leq 1 $,} \\
\searrow, & \text{if $1< r\leq 2$,}\\
0, & \text{if $r>  2$.}
\end{cases}
\]

Note that $\text{supp}(\varphi_{R})$ and $\text{supp}(\L\varphi_{R})$ are subsets of $\Omega_{1}$, where
\begin{equation}
\Omega_{1}:=\{(x,t)\in\H
\times \mathbb{R}_{+}:|\xi|^{4}+|\tilde{\xi}|^{4}+\tau^{2}+t^{2}\leq2R^{4}\}\,\,.
\end{equation}

By using \eqref{pseudoMh} and Young's inequality, from the identity  \eqref{pseudoweakheat} we obtain
\begin{equation}\label{pseudoheatocen1}
\begin{split}
\int_{\Q}|u|^p\varphi_{R} dx dt&
\leq C\left(A_{p}(\varphi_{R})+B_{p}(\varphi_{R})+\int_{\H}u_{0}|\L \varphi_{R}(x,0)|dx\right),
\end{split}
\end{equation}
where
\begin{equation}
\label{35}
\begin{split}
&A_{p}(\varphi_{R})=\int_{\Q_{1}}|\varphi_{R}|^{-\frac{1}{p-1}}(|(\varphi_{R})_{t}|^{\frac{p}{p-1}}+|\L\varphi_{R}|^{\frac{p}{p-1}})dxdt,
\\&B_{p}(\varphi_{R})=\int_{\Q_{1}}|\varphi_{R}|^{-\frac{1}{p-1}}|\L(\varphi_{R})_{t}|^{\frac{p}{p-1}}dxdt.
\end{split}
\end{equation}
Then  we have
\begin{equation*}
\left|\frac{\partial \varphi_{R}}{\partial t}\right|\leq CR^{-2},
\end{equation*}
and
\begin{equation*}
|\L (\varphi_{R})_{t}|\leq CR^{-4}.
\end{equation*}
Now, by using property of homogeneous dimension of $\H$ and by choosing variables $R^{2}\overline{t}=t$, $R\overline{\xi}=\xi$, $R\widehat{\xi}=\tilde{\xi}$ and $R^{2}\tilde{\tau}=\tau$ in the domain $\Omega_{1}$, from \eqref{35} we obtain
\begin{equation}
\begin{split}
&|A_{p}(\varphi_{R})|\leq CR^{Q+2-\frac{2p}{p-1}},
\\&|B_{p}(\varphi_{R})|\leq CR^{Q+2-\frac{4p}{p-1}}.
\end{split}
\end{equation}
Thus, we conclude
\begin{equation*}
\begin{split}
\int_{\Q}|u|^p\varphi_{R} dx dt&\leq \int_{\Q}|u|^p\varphi_{R} dx dt + \int_{\H} u_{0}(x)\varphi_{R}(x,0)dx
\\&\leq C\left(R^{Q+2-\frac{2p}{p-1}}+R^{Q+2-\frac{4p}{p-1}}+\int_{\H}u_{0}|\L \varphi_{R}(x,0)|dx\right)
\\&\leq C\left(R^{Q+2-\frac{2p}{p-1}}+\int_{\Omega_{1}}u_{0}|\L \varphi_{R}(x,0)|dx\right)
\\&\leq CR^{Q+2-\frac{2p}{p-1}}+CR^{-2}.
\end{split}
\end{equation*}

Finally, by taking $p<1+\frac{2}{Q}$ and letting $R\rightarrow\infty$, one concludes
$$
\int_\Omega |u|^p  dxdt\leq0,
$$
that is, $u=0$. This is a contradiction.

The case $p=1+\frac{2}{Q}$ can be dealt with H\"{o}lder's inequality instead of Young's inequality in \eqref{pseudoheatocen1}.
\end{proof}

As the last problem of this section, we consider the system of the Kirchhoff type pseudo-parabolic equations
\begin{equation}
\label{pseudoheatsys}
\begin{split}\begin{cases}
&u_{t} - M_{1}\left(t,\int\limits_{\H} |\na u|^{2}dx, \int\limits_{\H}|\na \omega|^{2}dx, \int\limits_{\H}|\L u|^{2}dx, \int\limits_{\H}|\L \omega|^{2}dx\right)\L u\\&{\,\,\,\,\,\,}-\L u_{t} = |\omega|^{q}, \,\, (x, t) \in\Q,\\&
\omega_{t} - M_{2}\left(t,\int\limits_{\H} |\na u|^{2}dx, \int\limits_{\H}|\na \omega|^{2}dx, \int\limits_{\H}|\L u|^{2}dx, \int\limits_{\H}|\L \omega|^{2}dx\right)\L \omega\\& {\,\,\,\,\,\,}-\L \omega_{t} = |u|^{p},   \,\, (x, t) \in\Q,\\&
u(x, 0) = u_0(x)\geq0, \,\,\,x\in\H,\\&
\omega(x,0)=\omega_{0}(x)\geq0, \,\,\,x\in\H,\end{cases}
\end{split}
\end{equation}
where  $\Q:= \H \times (0, T),$ $T>0$, $p,q>1$ and $M_{1},M_{2}:\mathbb{R}_{+}\times\mathbb{R}_{+}\times\mathbb{R}_{+}\rightarrow\mathbb{R}$ are  bounded functions such that
\begin{equation}
\label{pseudohsysus1}
0<M_{1}(\cdot,\cdot,\cdot,\cdot,\cdot)\leq C_{1},
\end{equation}
and
\begin{equation}
\label{pseudohsysus2}
0<M_{2}(\cdot,\cdot,\cdot,\cdot,\cdot)\leq C_{2}.
\end{equation}

We formulate a definition of the weak solution to \eqref{pseudoheatsys} in the following form.
\begin{defn}
We say that the pair of functions $(u, \omega)\in C^{1}(0,T;S^{2}_{2}(\H))\cap L^{p}(\Q)\times C^{1}(0,T;S^{2}_{2}(\H))\cap L^{q}(\Q)$ with $p,q>1$ is a weak solution of the Cauchy problem \eqref{pseudoheatsys} on $\Q$ with the initial data $(u(x, 0), \omega(x,0)) = (u_0, \omega_0) \in  L^1_{loc}(\H) \times L^1_{loc}(\H)$ and $u_0, \omega_0\geq0$, if\\
\begin{align*}
 &\int_{\Q} |\omega|^q \varphi dxdt +\int_{\H} u_0(x) \varphi(x,0) dx=-\int_{\Q} u \varphi_{t} dxdt  +\int_{\H}u_{0}(x)\L\varphi(x,0)dx\\&
 -\int_{\Q} M_{1}\left(t,\int\limits_{\H} |\na u|^{2}dx, \int\limits_{\H}|\na \omega|^{2}dx, \int\limits_{\H}|\L u|^{2}dx, \int\limits_{\H}|\L \omega|^{2}dx\right)u \L\varphi dxdt \\&+\int_{\Q}u\L\varphi_{t}dxdt,
\end{align*}
and\\
\begin{align*}
&\int_{\Q}|u|^p \psi dxdt +\int_{\H} \omega_0(x)\psi(x,0)dx
=-\int_{\Q} \omega \psi_{t} dxdt+\int_{\H}\omega_{0}(x)\L\psi(x,0)dx \\&
- \int_{\Q} M_{2}\left(t,\int\limits_{\H} |\na u|^{2}dx, \int\limits_{\H}|\na \omega|^{2}dx, \int\limits_{\H}|\L u|^{2}dx, \int\limits_{\H}|\L \omega|^{2}dx\right)w \L\psi dxdt\\&+\int_{\Q} w\L\psi_{t} dxdt,
\end{align*}
for any test functions $0\leq\psi,\varphi \in C^{2,1}_{x,t}(\Q)$.
\end{defn}

\begin{thm}
Assume that $p,q>1$. 
Let $Q$ be a homogeneous dimension of $\H$ such that  $Q\leq \max\{A_{1},A_{2}\},$ where
\begin{equation}
A_{1}=\frac{\frac{1}{pq}+\frac{1}{q}}{\frac{1}{2qp'}+\frac{1}{2q'}} \,\,\, \hbox{and} \,\,\, A_{2}=\frac{\frac{1}{pq}+\frac{1}{p}}{\frac{1}{2q'p}+\frac{1}{2p'}}.
\end{equation} 
Then the system \eqref{pseudoheatsys} does not admit nontrivial weak solution.
\end{thm}
\begin{proof}
Repeating the proofs similar to those of Theorems \ref{thmheatsysblow} and \ref{pseudoheatthm}, we conclude this theorem. \end{proof}

\section{Kirchhoff type hyperbolic equations and systems}

In this section we show a nonexistence result for the Kirchhoff type hyperbolic equations and systems on the Heisenberg group.

\subsection{Hyperbolic type equations}
We consider the equation
\begin{equation}\label{wave}
\begin{cases}
u_{tt} - M\left(t,\int_{\H}|\na u|^{2}dx,\int_{\H}|\L u|^{2}dx\right) \L u= |u|^{p}, \,\,(x, t) \in\Q,\\
u(x, 0) = u_0(x), \,\,\,x\in\H,\\
u_t(x, 0) = u_1(x), \,\,\,x\in\H,
\end{cases}
\end{equation}
where $\Q:= \H \times (0, T),$ $T>0$, $p > 1$ and $M:\mathbb{R}_{+}\times\mathbb{R}_{+}\times\mathbb{R}_{+}\rightarrow\mathbb{R}$ is satisfying the following condition
\begin{equation}\label{M}
0<M(t,\cdot,\cdot)\leq C_{0}t^{\beta}, \,\,\forall t>0,\,\,\beta>-2,\,\,\,\text{and}\,\,Q>\frac{2}{2+\beta}.
\end{equation}

Let us formulate a definition of the weak solution of the Cauchy problem \eqref{wave}.
\begin{defn}\label{defnweak}
We say that $u$ is a weak solution to \eqref{wave} on $\Omega$ with the initial data
$u_{1}, u_{0} \in L^1_{loc}(\H), $ if $ u\in C^{2}(0,T;S^{2}_{2}(\Omega))\cap L^p(\Q) $ and it satisfies
\begin{multline}
\int_{\Q}|u|^p\varphi dx dt + \int_{\H} (u_1(x)\varphi(x,0)-  u_{0}(x)\varphi_{t}(x,0))dx \\
=\int_{\Q} u \varphi_{tt} dx dt  - \int_{\Q} M\left(t,\int_{\H}|\na u|^{2}dx,\int_{\H}|\L u|^{2}dx\right) u\L \varphi dx dt,
\end{multline}
for any test function $0\leq\varphi\in C^{2,2}_{x,t}(\Q)$.
\end{defn}
\begin{thm}\label{thm1}
Let $u_1, u_{0} \in L^1(\H)$, and assume that $\int_{\H} u_1  dx\geq0.$ Moreover, assume that
the conditions in \eqref{M} hold. If
\begin{equation}
1<p\leq p_{c}= 1+\frac{4}{(2+\beta)Q-2},
\end{equation}
with $Q=2n+2$ the homogeneous dimension  of $\H$, then there exists no nontrivial  weak solution to the equation \eqref{wave}.
\end{thm}
\begin{proof}\label{wave4.2}
Let us consider the case $1<p<p_{c}$. By taking into account the property \eqref{M} and by Definition \ref{defnweak}, for any test function $\varphi$, we have
\begin{equation}\label{nacho}
\begin{split}
\int_\Omega |u|^p\varphi dx dt &+ \int_{\H} (u_1(x)\varphi(x,0)-  u_{0}(x)\varphi_{t}(x,0))dx  \\&
=\int_\Omega u \varphi_{tt} dx dt   - \int_\Omega M\left(t,\int_{\H}|\na u|^{2}dx,\int_{\H}|\L u|^{2} dx\right)u \L \varphi dx dt\\&
\leq\left|\int_\Omega u \varphi_{tt} dx dt   - \int_\Omega M\left(t, \int_{\H}|\na u|^{2}dx,\int_{\H}|\L u|^{2}dx\right)u \L \varphi dx dt\right|
\\&\leq \int_\Omega |u \varphi_{tt}| dx dt  + \int_\Omega\left|M\left(t, \int_{\H}|\na u|^{2} dx,\int_{\H}|\L u|^{2} dx\right)\right||u| |\L \varphi |dx dt
\\& \leq \int_\Omega |u \varphi_{tt}| dx dt  + C_{0}\int_\Omega t^{\beta}|u| |\L \varphi | dx dt.
\end{split}
\end{equation}
Let us choose a test function in the following form
\begin{equation}
\label{TestFunctionR}
\varphi_R (x, t) = \Phi \left(\frac{|\xi|^{4}+|\tilde{\xi}|^{4}+\tau^{2}}{R^{4}}\right)\Phi \left(\frac{t^{2}}{R^{\theta}}\right), \,\,(x,t)\in\H\times\mathbb{R}_{+},\,\,\theta>0,\,\,R>0,
\end{equation}
 with $\Phi\in C^{\infty}_{0}(\mathbb{R}_{+})$ as
$$
\Phi(r) =
\begin{cases}
1, & \text{if $ 0\leq r \leq 1 $,} \\
\searrow, & \text{if $1< r\leq 2$,}\\
0, & \text{if $r>  2$.}
\end{cases}
$$

By noting that
\begin{equation}\label{pott}
\frac{\partial\varphi_{R}(x,t)}{\partial t}=\frac{2 t}{R^{\theta}}\Phi \left(\frac{|\xi|^{4}+|\tilde{\xi}|^{4}+\tau^{2}}{R^{4}}\right)\Phi_{t}'\left(\frac{t^{2}}{R^{\theta}}\right),
\end{equation}
one has
\begin{equation}\label{vaprhi0}
\frac{\partial\varphi_{R}(x,0)}{\partial t}=0.
\end{equation}
We note that $\text{supp}(\varphi_{R})$ and $\text{supp}(\L\varphi_{R})$ are subsets of $\Omega_{1}=\Gamma\cup\tilde{\Gamma}$, where
\begin{equation}\label{wavedomain}
\Gamma:=\{x=(\xi,\tilde{\xi},\tau)\in\H: |\xi|^{4}+|\tilde{\xi}|^{4}+\tau^{2}\leq2R^{4}\},\,\,\tilde{\Gamma}:=\{t:0\leq t^{2}\leq 2R^{\theta}\}.
\end{equation}
We seek the test function $\varphi_{R}$ with the following properties
\begin{equation*}
 A_{p}(\varphi_{R}):=\int_{\Omega_{1}}\varphi_{R}^{-\frac{1}{p-1}} |(\varphi_{R})_{tt}|^{\frac{p}{p-1}} dx dt<\infty,
\end{equation*}
and
\begin{equation*}
 B_{p}(\varphi_{R}):=\int_{\Omega_{1}} t^{\frac{\beta p}{p-1}}\varphi_{R}^{-\frac{1}{p-1}} |\L \varphi_{R}|^{\frac{p}{p-1}} dx dt<
\infty. 
\end{equation*}
Then by  using Young's inequality, we get\\
\begin{equation*}
\begin{split}
    \int_{\Omega_{1}} |u| |(\varphi_{R})_{tt}| dx dt &= \int_{\Omega_{1}} u\varphi_{R}^{\frac{1}{p}}\varphi_{R}^{-\frac{1}{p}}|(\varphi_{R})_{tt}|dx dt\\&
\leq \varepsilon \int_{\Omega_{1}} |u|^p \varphi_{R} dx dt +C(\varepsilon) \int_{\Omega_{1}}\varphi_{R}^{-\frac{1}{p-1}} |(\varphi_{R})_{tt}|^\frac{p}{p-1} dx dt,
\end{split}
\end{equation*}
and
\begin{equation*}
\begin{gathered}
\int_{\Omega_{1}}t^{\beta}|u| |\L \varphi_{R}| dt dx \leq \varepsilon \int_{\Omega_{1}}|u|^p \varphi_{R} dt dx+ C(\varepsilon) \int_{\Omega_{1}}t^{\frac{\beta p}{p-1}}\varphi_{R}^{-\frac{1}{p-1}} |\L \varphi_{R}|^{\frac{p}{p-1}} dx dt,
\end{gathered}
\end{equation*}
for some positive constant $C(\varepsilon)$.
Hence by using the last fact and \eqref{vaprhi0}, we obtain
\begin{equation}\label{obs}
\int_\Omega |u|^p \varphi_{R} dx dt + C_{1}\int_{\H} u_1(x)\varphi_{R}(x,0) dx
\leq C \left(A_p(\varphi_{R}) + B_p(\varphi_{R})\right),
\end{equation}
where $C,C_{1}>0$.

Then from \cite{PohVer}, one obtains
\begin{equation}
\begin{split}\label{poxx}
&\L\varphi_{R}(x,t)=\frac{4(n+4)}{R^{4}}\left(|\xi|^{2}+|\tilde{\xi}|^{2}\right)\Phi'\left(\frac{|\xi|^{4}+|\tilde{\xi}|^{4}+\tau^{2}}{R^{4}}\right)\Phi\left(\frac{t^{2}}{R^{2}}\right)
\\&+\frac{16}{R^{8}}\left((|\xi|^{6}+|\tilde{\xi}|^{6})+2\tau(|\xi|^{2}-|\tilde{\xi}|^{2})\xi\cdot\tilde{\xi}+\tau^{2}(|\xi|^{2}+|\tilde{\xi}|^{2})\right)
\\&\times\Phi''\left(\frac{|\xi|^{4}+|\tilde{\xi}|^{4}+\tau^{2}}{R^{4}}\right)\Phi\left(\frac{t^{2}}{R^{2}}\right).
\end{split}
\end{equation}
By choosing variables $R^{\frac{\theta}{2}}\overline{t}=t$, $R\overline{\xi}=\xi$, $R\widehat{\xi}=\tilde{\xi}$ and $R^{2}\tilde{\tau}=\tau$ in the domain $\Omega_{1}$, $\theta=\frac{2}{1+\frac{\beta}{2}}$ and \eqref{vaprhi0}, one calculates
\begin{equation}\label{ocenk4}
\begin{split}
\int_{\Omega} |u|^p \varphi_{R} dx dt &+ C_{1}\int_{\H} u_1(x) \varphi_R(x, 0) dx\\&
\stackrel{\eqref{obs}}\leq C\left(R^{-\theta p'+Q+\frac{\theta}{2}}+R^{-2p'+(\frac{\beta \theta}{2})p'+Q+\frac{\theta}{2}}\right)\\&
\stackrel{\theta=\frac{2}{1+\frac{\beta}{2}}}=2 CR^{-\theta p'+Q+\frac{\theta}{2}}.
\end{split}
\end{equation}
On the other hand, we have
\begin{multline*}
\int_\Omega |u|^p \varphi_R dx dt + C_{1}\int_{\H} u_1(x) \varphi_R(x, 0) dx\\
\geq \liminf_{R \to \infty} \int_\Omega |u|^p \varphi_R dx dt + C_{1}\liminf_{R \to \infty} \int_{\H} u_1(x) \varphi_R(x, 0) dx.
\end{multline*}
Using the monotone convergence theorem, we obtain
\[
\liminf_{R \to \infty} \int_\Omega |u|^p \varphi_R dx dt = \int_\Omega |u|^p dxdt.
\]
Since $u_1 \in  L^1(\H)$, by the dominated convergence theorem, one has
$$
\liminf_{R \to \infty} \int_{\H} u_1(x) \varphi_R(x, 0) dx=\int_{\H} u_1(x) dx.\\
$$
Now, we get
\[
\liminf_{R \to \infty} \left(\int_\Omega |u|^p \varphi_R dx dt + C_{1}\int_{\H} u_1(x) \varphi_R(x,0) dx\right) \geq \int_\Omega |u|^p  dx dt +C_{1}d,
\]
where
\[ 
d= \int_{\H} u_1(x) dx \geq 0. 
\]
Then, 
\begin{multline*}
\int_\Omega |u|^p \varphi_R dx dt  +C_{1}\int_{\H} u_1(x) \varphi_R(x,0) dx \\ \geq
\liminf_{R \to \infty} \left(\int_\Omega |u|^p \varphi_R dx dt  +C_{1}\int_{\H} u_1(x) \varphi_R(x,0) dx\right)\geq  \int_\Omega |u|^p  dx dt +C_{1}d.
\end{multline*}
From \eqref{ocenk4} and $u_{1}\in L^{1}(\H)$, it can be shown that
\begin{equation}\label{vkon2}
\int_\Omega |u|^p dx dt+C_{1}d \leq R^{-\theta p'+Q+\frac{\theta}{2}}.
\end{equation}
By the assumption $p<1+\frac{\theta}{Q-\frac{\theta}{2}}=p_{c}$,
and where we take $\theta=\frac{2}{1+\frac{\beta}{2}}>0$.
Finally, letting $R\rightarrow\infty,$ one has
$$
\int_\Omega |u|^p dx dt\leq\int_\Omega |u|^p dx dt+C_{1}d\leq0.
$$
That is a contradiction, concluding that
$$
\int_\Omega |u|^p dx dt\leq0.
$$

Let us consider the critical case $p =  1+\frac{\theta}{Q-\frac{\theta}{2}}$. By using inequality \eqref{vkon2}, we obtain
\begin{equation}
\int_\Omega |u|^p dx dt \leq C< \infty,
\end{equation}
and
\begin{equation}
\lim_{R \to \infty} \int_{\overline{\Omega}_R} |u|^p \varphi_R dx dt=0,
\end{equation}
where $\overline{\Omega}_R=\Gamma_R\cup\tilde{\Gamma}_R,$
\begin{equation*}
\Gamma_R:=\{x=(\xi,\tilde{\xi},\tau)\in\H: R^{4}\leq|\xi|^{4}+|\tilde{\xi}|^{4}+\tau^{2}\leq2R^{4}\},\,\,\tilde{\Gamma}_R:=\{t:R^{\theta}\leq t^{2}\leq 2R^{\theta}\}.
\end{equation*}
Using the H\"{o}lder inequality instead of Young's inequality in \eqref{nacho}, one has
\[
\int_\Omega |u|^p \varphi_R dx dt\leq \int_\Omega |u|^p \varphi_R dx dt + d \leq C\left(\int_{\overline{\Omega}_R}  |u|^p \varphi_R dx dt\right)^{\frac{1}{p}}.
\]
Letting $R \to \infty$, we get
\[
\int_\Omega |u|^p \varphi_R dx dt  =0,
\]
completing the proof.
\end{proof}

\subsection{Hyperbolic Kirchhoff type systems}

Let us consider the Kirchhoff type system of wave equations in $\Q=\H \times (0, T),$ $T>0$ with the Cauchy data
\begin{equation}
\label{wavesys}
\begin{cases}
u_{tt} - M_{1}\left(t,\int\limits_{\H}|\na u|^{2}dx,\int\limits_{\H}|\na v|^{2}dx,\int\limits_{\H}|\L u|^{2}dx, \int\limits_{\H}|\L v|^{2}dx\right)\L u = |v|^{q},\\
v_{tt} - M_{2}\left(t,\int\limits_{\H}|\na u|^{2}dx,\int\limits_{\H}|\na v|^{2}dx,\int\limits_{\H}|\L u|^{2}dx, \int\limits_{\H}|\L v|^{2}dx\right)\L v = |u|^{p},\\
u(x, 0) = u_0(x), u_t(x, 0) = u_1(x), \,\,\,x\in\H,\\
v(x,0)=v_{0}(x), v_{t}(x,0)=v_{1}(x), \,\,\,x\in\H,
\end{cases}
\end{equation}
where $p,q>1$ and $M_{1},M_{2}:\mathbb{R}_{+}\times\mathbb{R}_{+}\times\mathbb{R}_{+}\rightarrow\mathbb{R}$ are satisfying the following conditions
\begin{equation}\label{sysus1}
0<M_{i}(t,\cdot,\cdot,\cdot,\cdot)\leq C_{i}t^{\beta_{i}},\,\,\forall t>0,\,\,\beta_{i}>-2,\,\,\,\text{and}\,\,Q>\frac{4}{2+\beta_{i}},
\end{equation}
for $i=1, 2$.

Now, we formulate a definition of the weak solution to \eqref{wavesys}.
\begin{defn}
We say that the pair $(u, v)$ is a weak solution to the equation \eqref{wavesys} on $\Q$ with the initial data $(u(x, 0), v(x,0)) = (u_0, v_0) \in  L^1_{loc}(\H) \times L^1_{loc}(\H)$, if $(u, v) \in C^{2}(0,T;S^{2}_{2}(\H))\cap L^p(\Q) \times C^{2}(0,T;S^{2}_{2}(\H))\cap L^q(\Q) $ satisfies\\
\begin{align*}
&\int_{\Q}|v|^q \varphi dx dt +\int_{\H} u_1(x)\varphi(x,0) dx-\int_{\H} u_0(x) \varphi_{t}(x,0) dx
=\int_{\Q} u \varphi_{tt} dx dt \\&-\int_{\Q} M_{1}\left(t,\int\limits_{\H}|\na u|^{2}dx,\int\limits_{\H}|\na v|^{2}dx,\int\limits_{\H}|\L u|^{2}dx, \int\limits_{\H}|\L v|^{2}dx\right)u \L\varphi dx dt,
\end{align*}
and
\begin{align*}
&\int_{\Q}|u|^p \psi dx dt +\int_{\H} v_1(x)\psi(x,0) dx-\int_{\H} v_0(x)\psi_{t}(x,0)dx
=\int_{\Q} v \psi_{tt} dx dt\\& - \int_{\Q} M_{2}\left(t,\int\limits_{\H}|\na u|^{2}dx,\int\limits_{\H}|\na v|^{2}dx,\int\limits_{\H}|\L u|^{2}dx, \int\limits_{\H}|\L v|^{2}dx\right)v \L\psi dx dt,
\end{align*}
for any test function $0\leq\varphi,\psi\in C^{2,2}_{x,t}(\Omega)$.
\end{defn}

For simplicity, let us denote
\begin{equation}
\theta_{i}=\frac{4}{2+\beta_{i}},\,\,\,\,\,i=1,2,
\end{equation}
and $\widetilde{\theta}=\max\{\theta_{1},\theta_{2}\}$. We give the following main result for the system of wave equations.
\begin{thm}
\label{thm2}
Suppose that the condition \eqref{sysus1} holds. Assume that the initial data $(u_0, v_0), (u_1, v_1) \in L^1(\H) \times L^1(\H)$ satisfy
\begin{equation}
\int_{\H} u_1 dx\geq0,\,\,\,\text{and}\,\,\,\int_{\H} v_1 dx\geq0.
\end{equation}
If
\begin{equation*}
1 < pq \leq  1 + \frac {\theta_{2}q+\theta_{1}}{Q+\frac{\widetilde{\theta}}{2}-\theta_{1}}, \,\,\, \hbox{or} \,\,\, 1 < pq \leq  1 + \frac {\theta_{1}p+\theta_{2}}{Q+\frac{\widetilde{\theta}}{2}-\theta_{2}},
\end{equation*}
with $Q=2n+2$ the homogeneous dimension of $\H$, then there exists no nontrivial weak solution of the system \eqref{wavesys}.
\end{thm}
\begin{proof} 
 By analogy with the single equation case, one obtains
\begin{equation*}
\begin{split}
     \int_\Omega |v|^q \varphi dx dt &+
 \int_{\H} u_1(x) \varphi(x,0) dx
\leq \int_\Omega|u| |\varphi_{tt}| dx dt
+ C_{1}\int_\Omega t^{\beta_{1}}|u| |\L\varphi| dx dt
\end{split}
\end{equation*}
and
\begin{equation*}
 \int_\Omega |u|^p \psi dx dt +\int_{\H} v_1(x) \psi(x,0) dx
\leq\int_\Omega |v| |\psi_{tt}| dx dt + C_{2}\int_{\Omega}t^{\beta_{2}} |v| |\L\psi| dx dt.
\end{equation*}
By taking $\varphi = \varphi_R$ and $\psi=\psi_{R}$ as in \eqref{TestFunctionR} and using the H\"{o}lder inequality, we calculate
\begin{multline*}
\int_\Omega |v|^q \varphi_R dxdt +
 \int_{\H} u_1(x) \varphi_{R}(x,0) dx \\ 
 \leq C(A_p(\varphi_R)^{\frac{p-1}{p}} + B_p(\varphi_R)^{\frac{p-1}{p}})\left( \int_\Omega|u|^p \psi_R dxdt\right)^{\frac{1}{p}},
\end{multline*}
and 
\begin{multline*}
\int_\Omega |u|^p \psi_R dxdt +
 \int_{\H} v_1(x) \psi_{R}(x,0) dx\\
 \leq C(A_q(\psi_R)^{\frac{q-1}{q}} + B_q(\psi_R)^{\frac{q-1}{q}})\left( \int_\Omega|v|^q \varphi_R dxdt\right)^{\frac{1}{q}},
\end{multline*}
where $C>0$,
\begin{equation*}
A_p(\varphi_{R}) = \int_{\Omega_{1}}\varphi_{R}^{-\frac{1}{p-1}} |(\varphi_{R})_{tt}|^{\frac{p}{p-1}} dx dt<\infty,
\end{equation*}
\begin{equation*}
B_p(\varphi_{R}) = \int_{\Omega_{1}} t^{\frac{\beta_{1} p}{p-1}}\varphi_{R}^{-\frac{1}{p-1}} |\L \varphi_{R}|^{\frac{p}{p-1}} dx dt<\infty, 
\end{equation*}
\begin{equation*}
A_q(\psi_{R}) = \int_{\Omega_{1}}\psi_{R}^{-\frac{1}{q-1}} |(\psi_{R})_{tt}|^{\frac{q}{q-1}} dx dt<\infty,
\end{equation*}
and
\begin{equation*}
B_q(\psi_{R}) = \int_{\Omega_{1}} t^{\frac{\beta_{2} p}{p-1}}\psi_{R}^{-\frac{1}{q-1}} |\L \psi_{R}|^{\frac{q}{q-1}} dx dt<\infty. 
\end{equation*}

By using the dominated convergence theorem with $u_{1}, v_{1}\in L^{1}_{loc}(\H)$, and conditions $\int_{\H}u_{1}dx\geq0$ and $\int_{\H}v_{1}dx\geq0$, we obtain
\begin{equation*}
\begin{split}
 \int_{\H} u_1(x) \varphi_R(x,0) dx \geq \liminf_{R\rightarrow\infty}\int_{\H} u_1(x) \varphi_R(x,0) dx=\int_{\H}u_{1}dx\geq0 ,\\
  \int_{\H} v_1(x) \psi_R(x,0) dx \geq \liminf_{R\rightarrow\infty}\int_{\H} v_1(x) \psi_R(x,0) dx=\int_{\H}v_{1}dx\geq0.\\
\end{split}
\end{equation*}
Thus, we get
\begin{equation*}
\int_\Omega |v|^q \varphi_R dxdt \\
\leq C(A_p(\varphi_R)^{\frac{p-1}{p}} + B_p(\varphi_R)^{\frac{p-1}{p}})\left( \int_\Omega |u|^p \psi_R dxdt\right)^{\frac{1}{p}},
\end{equation*}
and
\begin{equation*}
\int_\Omega |u|^p \psi_R dxdt
\leq C(A_q(\psi_R)^{\frac{q-1}{q}} + B_q(\psi_R)^{\frac{q-1}{q}})\left( \int_\Omega|v|^q \varphi_R dxdt\right)^{\frac{1}{q}}.
\end{equation*}

By choosing variables $R^{\frac{\theta_{i}}{2}}\overline{t}=t$, $R\overline{\xi}=\xi$, $R\widehat{\xi}=\tilde{\xi}$ and $R^{2}\tilde{\tau}=\tau$ in the domain $\Omega_{1}$ (in \eqref{wavedomain}), for $i=1,2$, we get
\begin{equation*}
\begin{split}
\int_\Omega |v|^q \varphi_R dxdt
\leq CR^{\frac{\left(Q+\frac{\theta_{1}}{2}\right)(p-1)}{p}-\theta_{1}}\left( \int_\Omega|u|^p \psi_R dxdt\right)^{\frac{1}{p}},
 \end{split}
\end{equation*}
and
\begin{equation*}
\begin{split}
\int_\Omega |u|^p \psi_R dxdt
\leq CR^{\frac{\left(Q+\frac{\theta_{2}}{2}\right)(q-1)}{q}-\theta_{2}}\left( \int_\Omega|v|^q \varphi_R dxdt\right)^{\frac{1}{q}}.
 \end{split}
\end{equation*}
Combining the last two facts, one concludes
\begin{equation}\label{syspse1}
\begin{split}
\left(\int_\Omega |u|^p \psi_R dxdt\right)^{1-\frac{1}{pq}}  \leq CR^{\alpha_1},
 \end{split}
\end{equation}
and
\begin{equation}\label{syspse2}
\begin{split}
\left( \int_\Omega |v|^q \varphi_R dxdt\right)^{1-\frac{1}{pq}}  \leq  CR^{\alpha_2},
 \end{split}
\end{equation}
where
\[
\alpha_1 = \frac{\left(Q+\frac{\widetilde{\theta}}{2}\right)(pq-1)-\theta_{1}pq-\theta_{2}q}{pq} \,\,\,  \hbox{and} \,\,\, \alpha_{2}=\frac{\left(Q+\frac{\widetilde{\theta}}{2}\right)(pq-1)-\theta_{2}pq-\theta_{1}p}{pq},
\]
with $\widetilde{\theta}=\max\{\theta_{1},\theta_{2}\}$.


\textbf{Case 1:} $1 < pq < 1 + \frac {\theta_{2}q+\theta_{1}}{Q+\frac{\widetilde{\theta}}{2}-\theta_{1}},$ $\left(\text{or}\,\,1 < pq <1+\frac {\theta_{1}p+\theta_{2}}{Q+\frac{\widetilde{\theta}}{2}-\theta_{2}}\right)$. By letting $R \to \infty$ in \eqref{syspse1} with $1 < pq < 1 + \frac {\theta_{2}q+\theta_{1}}{Q+\frac{\widetilde{\theta}}{2}-\theta_{1}}$, we obtain
\[
\int_\Omega |u|^p  dxdt =0,
\]
which is a contradiction. Similarly,  from \eqref{syspse2} and $1 < pq <\frac {\theta_{1}p+\theta_{2}}{Q+\frac{\widetilde{\theta}}{2}-\theta_{2}}$, we obtain
\[
\int_\Omega |v|^q  dxdt =0.
\]

\textbf{Case 2:} $pq =  1 +  \max\left\{\frac {\theta_{2}q+\theta_{1}}{Q+\frac{\widetilde{\theta}}{2}-\theta_{1}},\frac {\theta_{1}p+\theta_{2}}{Q+\frac{\widetilde{\theta}}{2}-\theta_{2}}\right\}$. This case can be treated in the same way
as in the proof of Theorem \ref{thm1}.
\end{proof}

\begin{cor}
 If $p = q>1$,  $u = v$  and $M_{1}=M_{2}$ in Theorem \ref{thm2}, we obtain the result for a single
equation given by Theorem \ref{thm1}.
\end{cor}
\begin{proof}
From Theorem \ref{thm2}, we get
$$p^{2}\leq1+\frac{\frac{4}{2+\beta}(p+1)}{Q-\frac{2}{2+\beta}},$$
which means
$$p^{2}-1\leq\frac{4(p+1)}{(2+\beta)Q-2}.$$ Thus, by dividing both sides by $p+1$, we obtain
\begin{equation}
p-1\leq \frac{4}{(2+\beta)Q-2},
\end{equation}
finishing the proof.
\end{proof}

\subsection{Pseudo-hyperbolic equations and systems}

In this subsection we show a nonexistence result for the Kirchhoff type pseudo-hyperbolic equation on Heisenberg groups in the following form:
\begin{equation}\label{pseudowave}
\begin{cases}
 u_{tt} -M\left(t,\int_{\H}|\na u|^{2}dx,\int_{\H}|\L u|^{2}dx\right) (\L u+\L u_{tt})= |u|^{p},   \,\, (x, t) \in
 \Q,\\
u(x, 0) = u_0(x), u_t(x, 0) = u_1(x), \,\,\,x\in\H,
\end{cases}
\end{equation}
where $p > 1$, $\Q= \H \times (0, T)$ and $M:\mathbb{R}_{+}\times\mathbb{R}_{+}\times\mathbb{R}_{+}\rightarrow\mathbb{R}$ is a bounded function such that
\begin{equation}\label{pseudoM}
0<M(\cdot,\cdot,\cdot)\leq C_{0}.
\end{equation}
Let us formulate a definition of the weak solution of the equation \eqref{pseudowave}.
\begin{defn}
\label{pseudodefnweak}
We say that $u$ is a weak solution to \eqref{pseudowave} on $\H\times(0,T)$ with $u_{0}, u_{1} \in L^1_{loc}(\H), $ if $ u\in C^{2}(0,T;S^{2}_{2}(\H))\cap L^p(\Q) $ and satisfies\\
\begin{equation}
\begin{split}
&\int_{\Q}|u|^p\varphi dx dt + \int_{\H} u_1(x)(\varphi(x,0)-\L \varphi(x,0))dx
\\&-  \int_{\H} u_{0}(x)(\varphi_{t}(x,0)-\L \varphi_{t}(x,0))dx
=\int_{\Q} u \varphi_{tt} dx dt  \\&
- \int_{\Q} M\left(t,\int_{\H}|\na u|^{2}dx,\int_{\H}|\L u|^{2}dx\right)u(\L \varphi +\L \varphi_{tt}) dx dt,
\end{split}
\end{equation}
for any test function  $0\leq\varphi\in C^{2,2}_{x,t}(\Q)$.
\end{defn}

\begin{thm}
\label{pseudothm1}
Assume that $u_1, u_{0} \in L^1(\H)$, and
\begin{equation}
\int_{\H} u_1  dx\geq0.
\end{equation}
Suppose that the condition \eqref{pseudoM} is true. If
\begin{equation}
1<p\leq p_{c}= \frac{Q+1}{Q-1},
\end{equation}
with $Q$ the homogeneous dimension of $\H$, then  there exists no nontrivial weak solution to \eqref{pseudowave}.
\end{thm}
\begin{proof}
Consider the case $1<p<p_{c}$. 
Let us choose a test function in the following form
\begin{equation}
\label{TestFunctionR1}
\varphi_R (x, t) = \Phi \left(\frac{|\xi|^{4}+|\tilde{\xi}|^{4}+\tau^{2}}{R^{4}}\right)\Phi \left(\frac{t^{2}}{R^{2}}\right), \,\,(x,t)\in\H\times\mathbb{R}_{+},\,\,\,\,R>0,
\end{equation}
 with $\Phi\in C^{\infty}_{0}(\mathbb{R}_{+})$ as
$$
\Phi(r) =
\begin{cases}
1, & \text{if $ 0\leq r \leq 1 $,} \\
\searrow, & \text{if $1< r\leq 2$,}\\
0, & \text{if $r>  2$.}
\end{cases}
$$

By noting that
\begin{equation}\label{pott1}
\frac{\partial\varphi_{R}(x,t)}{\partial t}=\frac{2 t}{R^{2}}\Phi \left(\frac{|\xi|^{4}+|\tilde{\xi}|^{4}+\tau^{2}}{R^{4}}\right)\Phi_{t}'\left(\frac{t^{2}}{R^{2}}\right),
\end{equation}
one has
\begin{equation}\label{vaprhi01}
\frac{\partial\varphi_{R}(x,0)}{\partial t}=0.
\end{equation}
By Definition \ref{pseudodefnweak}, we obtain
\begin{equation}
\label{nacho123}
\begin{split}
&\int_\Omega |u|^p\varphi_{R} dx dt + \int_{\H} u_1(x)\varphi_{R}(x,0)dx\stackrel{\eqref{vaprhi01}}=
\int_\Omega |u|^p\varphi_{R} dx dt \\&+ \int_{\H} u_1(x)\varphi_{R}(x,0)dx
-  \int_{\H} u_{0}(x)((\varphi_{R})_{t}(x,0)-\L (\varphi_{R})_{t}(x,0))dx\\&
=\int_\Omega u (\varphi_{R})_{tt} dx dt  - \int_{\Q} M\left(t,\int_{\H}|\na u|^{2}dx,\int_{\H}|\L u|^{2}dx\right)u(\L \varphi +\L \varphi_{tt}) dx dt\\&
+\int_{\H} u_1(x)\L \varphi_{R}(x,0)dx\\&
\leq\big{|}\int_\Omega u (\varphi_{R})_{tt} dx dt  - \int_{\Q} M\left(t,\int_{\H}|\na u|^{2}dx,\int_{\H}|\L u|^{2}dx\right)u(\L \varphi +\L \varphi_{tt}) dx dt\\&
+\int_{\H} u_1(x)\L \varphi_{R}(x,0)dx\big{|}\\&
\leq \int_\Omega |u (\varphi_{R})_{tt}| dx dt  + C_{0}\int_\Omega|u\L (\varphi_{R})_{tt}| dx dt + C_{0}\int_\Omega |u \L \varphi_{R}| dx dt\\&
+\int_{\H} |u_1(x)\L \varphi_{R}(x,0)|dx.
\end{split}
\end{equation}

Then by  using the Young inequality, we get
\begin{equation}
\begin{split}
\int_{\Omega_{1}} |u||(\varphi_{R})_{tt}| dxdt&
=\int_{\Omega_{1}}|u|\varphi_{R}^{\frac{1}{p}}\varphi_{R}^{-\frac{1}{p}}|(\varphi_{R})_{tt}|dxdt
\\&
\leq \varepsilon \int_{\Omega_{1}}|u|^p \varphi_{R} dxdt +C(\varepsilon) \int_{\Omega_{1}}\varphi_{R}^{-\frac{1}{p-1}} |(\varphi_{R})_{tt}|^\frac{p}{p-1} dxdt,
\end{split}
\end{equation}
\begin{equation}
\begin{gathered}
\int_{\Omega_{1}}u |\L \varphi_{R}| dt dx \leq \varepsilon \int_{\Omega_{1}}|u|^p \varphi_{R} dt dx+ C(\varepsilon) \int_{\Omega_{1}}\varphi_{R}^{-\frac{1}{p-1}} |\L \varphi_{R}|^{\frac{p}{p-1}} dxdt,
\end{gathered}
\end{equation}
and
\begin{equation}
\begin{gathered}
\int_{\Omega_{1}}u |\L (\varphi_{R})_{tt}| dt dx \leq \varepsilon \int_{\Omega_{1}}|u|^p \varphi_{R} dt dx+ C(\varepsilon) \int_{\Omega_{1}}\varphi_{R}^{-\frac{1}{p-1}} |\L (\varphi_{R})_{tt}|^{\frac{p}{p-1}} dxdt,
\end{gathered}
\end{equation}
for some positive constant $C(\varepsilon)$.
By using the above facts, we obtain
\begin{equation}
\label{pseudoobs}
\begin{split}
\int_\Omega |u|^p \varphi_{R} dxdt &+ C\int_{\H} u_1(x)\varphi_{R}(x,0) dx
\\&\leq C \left(A_p(\varphi_{R}) + B_p(\varphi_{R}) + \int_{\H} |u_1(x) | \L \varphi_{R} (x,0)| dx\right),
\end{split}
\end{equation}
where $C>0$,
\begin{equation}
A_p(\varphi_{R}) = \int_{\Omega_{1}}\varphi_{R}^{-\frac{1}{p-1}} |(\varphi_{R})_{tt}|^{\frac{p}{p-1}} dx dt<\infty,
\end{equation}
and
\begin{equation}
B_p(\varphi_{R}) = \int_{\Omega_{1}}\varphi_{R}^{-\frac{1}{p-1}} |\L (\varphi_{R})_{tt}|^{\frac{p}{p-1}} dxdt<\infty.
\end{equation}

From \eqref{vaprhi0}, we have
\begin{equation}
\frac{\partial\varphi_{R}(x,0)}{\partial t}=0.
\end{equation}
Also, let us denote $\overline{\Omega}_R=\Gamma_R\cup\tilde{\Gamma}_R,$
\begin{equation*}
\Gamma_R:=\{x=(\xi,\tilde{\xi},\tau)\in\H: R^{4}\leq|\xi|^{4}+|\tilde{\xi}|^{4}+\tau^{2}\leq2R^{4}\},\,\,\tilde{\Gamma}_R:=\{t:R^{2}\leq t^{2}\leq 2R^{2}\}.
\end{equation*}
Let us estimate $A_{p}(\varphi_{R}), B_{p}(\varphi_{R})$. By using the homogeneous dimension of the Heisenberg groups, we calculate
\begin{equation}\label{pseudoocen1}
|A_{p}(\varphi_{R})|
\leq C R^{Q+1-\frac{2p}{p-1}},
\end{equation}
and
\begin{equation}\label{pseudoocen2}
|B_{p}(\varphi_{R})|\leq C R^{Q+1-\frac{4p}{p-1}}.
\end{equation}
Choose variables $R\overline{t}=t$, $R\overline{\xi}=\xi$, $R\widehat{\xi}=\tilde{\xi}$ and $R^{2}\tilde{\tau}=\tau$ in the domain $\Omega_{1}$,  and by using \eqref{pseudoocen1},\eqref{pseudoocen2} in \eqref{pseudoobs}, one calculates
\begin{equation}
\label{pseudoocenk4}
\begin{split}
&\int_{\Omega} |u|^p \varphi_{R} dxdt + C\int_{\H} u_1(x) \varphi_R(x, 0) dx\\&
\leq C\left(R^{Q+1-\frac{2p}{p-1}}+C\int_{\Gamma} |u_1(x)| |\L \varphi_R(x, 0)|dx\right).
\end{split}
\end{equation}
On the other hand, we have
\begin{multline*}
\int_\Omega |u|^p \varphi_R dxdt + C\int_{\H} u_1(x) \varphi_R(x, 0) dx\\
\geq \liminf_{R \to \infty} \int_\Omega |u|^p \varphi_R dxdt + C\liminf_{R \to \infty} \int_{\H} u_1(x) \varphi_R(x, 0) dx.
\end{multline*}
Using the monotone convergence theorem, we obtain
\[
\liminf_{R \to \infty} \int_\Omega |u|^p \varphi_R dxdt = \int_\Omega |u|^p dxdt.
\]
Since $u_1 \in  L^1(\mathbb{G})$, by the dominated convergence theorem, one has
$$
\liminf_{R \to \infty} \int_{\Gamma} u_1(x) \varphi_R(x, 0) dx=\int_{\H} u_1(x) dx.
$$
Now, we have
\[
\liminf_{R \to \infty} \left(\int_\Omega|u|^p \varphi_R dxdt  +C\int_{\H} u_1(x) \varphi_R(x,0) dx\right) \geq \int_\Omega |u|^p  dxdt +Cd,
\]
where
\[ 
d= \int_{\H} u_1(x) dx \geq0. 
\]
By the definition of the limit, for every $\alpha > 0$ there exists $R_0 > 0$ such that
\begin{align*}
\int_\Omega |u|^p \varphi_R dx dt  &+C\int_{\H} u_1(x) \varphi_R(x,0) dx \\& \geq
\liminf_{R \to \infty} \left(\int_\Omega |u|^p \varphi_R dx dt  +\int_{\H} u_1(x) \varphi_R(x,0) dx\right)-\alpha\\&\geq  \int_\Omega |u|^p  dx dt +Cd -\alpha,
\end{align*}
for every $R \geq R_0$. By choosing $\alpha = \frac{Cd}{2},$ we get
\[
\int_\Omega |u|^p \varphi_R dxdt  +\int_{\H} u_1(x) \varphi_R(x,0) dx\geq  \int_\Omega |u|^p dxdt+\frac{Cd}{2},
\]
for every $R \geq R_0$. Then from $u_{1}\in L^{1}(\H)$, one obtains
\begin{equation}\label{pseudovkon2}
\begin{split}
\int_\Omega |u|^p dxdt+\frac{Cd}{2}&\leq C\left(R^{Q+1-\frac{2p}{p-1}}+\int_{\Gamma} |u_1(x)| |\L \varphi_R(x, 0)|dx\right)
\\&\leq C\left(R^{Q+1-\frac{2p}{p-1}}+R^{-2}\int_{\Gamma} |u_1(x)| dx\right)
\\&\leq C\left(R^{Q+1-\frac{2p}{p-1}}+R^{-2}\right).
\end{split}
\end{equation}
From assumption we have $p<1+\frac{2}{p-1}$, or, in other words $Q+1-\frac{2p}{p-1}<0$ and by using this fact with 
 letting $R\rightarrow\infty,$ we obtain
$$
\int_\Omega |u|^p dxdt\leq \int_\Omega |u|^p dxdt+C\frac{d}{2}\leq0,
$$
that is,
$$
\int_\Omega |u|^p dxdt=0.
$$
It is a contradiction. 

Now, we consider the critical case $p = 1 + \frac{2}{Q-1}$. By using \eqref{pseudovkon2}, we have
\begin{equation}
\int_\Omega |u|^p dxdt \leq C< \infty.
\end{equation}
Then
\begin{equation}
\lim_{R \to \infty} \int_{\overline{\Omega}_R} |u|^p \varphi_R dx dt=0,
\end{equation}
where $\overline{\Omega}_R=\Gamma_{1,R}\cup\tilde{\Gamma}_{1,R},$
\begin{equation*}
\Gamma_{1,R}:=\{x=(\xi,\tilde{\xi},\tau)\in\H: R^{4}\leq|x|^{4}\leq2R^{4}\},\,\,\tilde{\Gamma}_{1,R}:=\{t:R^{2}\leq t^{2}\leq 2R^{2}\}.
\end{equation*}
Using the H\"{o}lder inequality instead of the Young's inequality in \eqref{nacho123}, one obtains
\[
\int_\Omega |u|^p \varphi_R dxdt\leq \int_\Omega |u|^p \varphi_R dxdt + C\frac{d}{2}  \leq C\left(\int_{\overline{\Omega}_R}  |u|^p \varphi_R dxdt\right)^{\frac{1}{p}}.
\]
Letting $R \to \infty$ in the above, we get
\[
\int_\Omega |u|^p \varphi_R dxdt =0,
\]
arriving at a contradiction.
\end{proof}

Let us consider the Cauchy problem for the system of Kirchhoff type pseudo-hyperbolic equations
\begin{equation}
\label{pseudowavesys}
\begin{split}\begin{cases}
u_{tt} - M_{1}\left(t,\int\limits_{\H}|\na u|^{2}dx,\int\limits_{\H}|\na v|^{2}dx,\int\limits_{\H}|\L u|^{2}dx, \int\limits_{\H}|\L v|^{2}dx\right)(\L u+\L u_{tt})\\{\,\,\,\,\,} = |v|^{q},   \,\, (x, t) \in \Q,\\
v_{tt} - M_{2}\left(t,\int\limits_{\H}|\na u|^{2}dx,\int\limits_{\H}|\na v|^{2}dx,\int\limits_{\H}|\L u|^{2}dx, \int\limits_{\H}|\L v|^{2}dx\right)(\L v+\L v_{tt})\\{\,\,\,\,\,} = |u|^{p},   \,\, (x, t) \in \Q,\\
u(x, 0) = u_0(x), u_t(x, 0) = u_1(x), \,\,\,x\in\H,\\
v(x,0)=v_{0}(x), v_{t}(x,0)=v_{1}(x), \,\,\,x\in\H,
\end{cases}\end{split}
\end{equation}
where $\Q=\H \times (0, T),$ $T>0,$ $p,q>1$, and $M_{1},M_{2}:\mathbb{R}_{+}\times\mathbb{R}_{+}\times\mathbb{R}_{+}\rightarrow\mathbb{R}$ are bounded functions such that
\begin{equation}\label{pseudosysus1}
0<M_{1}(\cdot,\cdot,\cdot,\cdot,\cdot)\leq C_{1},
\end{equation}
and
\begin{equation}\label{pseudosysus2}
0<M_{2}(\cdot,\cdot,\cdot,\cdot,\cdot)\leq C_{2}.
\end{equation}

Now, we formulate the following definition of the weak solution to the system \eqref{pseudowavesys}.
\begin{defn}
We say that the pair $(u, v)$ is a weak solution to \eqref{pseudowavesys} on $\Q=\H\times(0,T)$ with the initial data $(u(x, 0), v(x,0)) = (u_0, v_0) \in  L^1_{loc}(\H) \times L^1_{loc}(\H)$, if $(u, v) \in  C^{2}(0,T;S^{2}_{2}(\H))\cap L^p(\Q) \times C^{2}(0,T;S^{2}_{2}(\H))\cap L^q(\Q)$ with $p,q>1$ and they satisfy
\begin{align*}
&\int_{\Q} |v|^q \varphi dxdt +\int_{\H} u_1(x)(\varphi(x,0)-\L\varphi(x,0)) dx\\&-\int_{\H} u_0(x)( \varphi_{t}(x,0)-\L\varphi_{t}(x,0)) dx
=\int_{\Q} u \varphi_{tt} dxdt \\&
- \int_{\Q} M_{1}\left(t,\int\limits_{\H}|\na u|^{2}dx,\int\limits_{\H}|\na v|^{2}dx,\int\limits_{\H}|\L u|^{2}dx, \int\limits_{\H}|\L v|^{2}dx\right)u(\L\varphi+ \L\varphi_{tt} )dxdt,
\end{align*}
and
\begin{align*}
&\int_{\Q} |u|^p \psi dxdt +\int_{\H} v_1(x)( \psi(x,0)-\L\psi(x,0)) dx\\&-\int_{\H} v_0(x)( \psi_{t}(x,0)-\L\psi_{t}(x,0)) dx
=\int_{\Q} v \psi_{tt} dxdt \\&
-\int_{\Q}M_{2}\left(t,\int\limits_{\H}|\na u|^{2}dx,\int\limits_{\H}|\na v|^{2}dx,\int\limits_{\H}|\L u|^{2}dx, \int\limits_{\H}|\L v|^{2}dx\right) v (\L\psi+\L\psi_{tt}) dxdt,
\end{align*}
for any test functions $0\leq\varphi,\psi\in C^{2,2}_{x,t}(\Omega)$.
\end{defn}

Then we have the following theorem.
\begin{thm}
\label{pseudothm2}
Let $p,q>1,$ $(u_0, v_0),(u_1, v_1) \in L^1(\H) \times L^1(\H)$ with
\begin{equation}
\int_{\H} u_1 dx\geq0,\,\,\,\text{and}\,\,\,\int_{\H} v_1 dx\geq0.
\end{equation}
Assume that the conditions \eqref{pseudosysus1} and \eqref{pseudosysus2} hold.
If
\[
1 < pq \leq 1+ \frac 2{Q-1} \max\{p + 1; q + 1\},
\]
with homogeneous dimension $Q=2n+2$ of $\H$, then there exists no nontrivial  weak solution to \eqref{pseudowavesys}.
\end{thm}
\begin{proof}The proof of this theorem is similar to the proof of Theorem \ref{pseudothm1}. The only thing that we need to note is that for the large $R$, we have
\begin{equation}
\begin{split}
&\int_{\H}u_{1}(x)\varphi_{R}(x,0)dx-\int_{\H}|u_{1}(x)|\L \varphi_{R}(x,0)|dx\geq0,
\\&\int_{\H}v_{1}(x)\psi_{R}(x,0)dx-\int_{\H}|v_{1}(x)|\L \psi_{R}(x,0)|dx\geq0.
\end{split}
\end{equation}
And, by repeating the proof of Theorem \ref{pseudothm1}, we finish our proof.
\end{proof}

\end{document}